\documentclass[11pt]{amsart}
\usepackage{lipsum}
\usepackage{amsfonts}
\usepackage{amsmath}
\usepackage{scalerel}
\usepackage{subcaption}
\usepackage{amsthm}
\usepackage{epsf,latexsym}
\usepackage{amsfonts}
\usepackage{fontawesome}
\usepackage{tikz}
\usetikzlibrary{arrows,matrix,positioning}
\usetikzlibrary{cd}
\usepackage{etex}
\usepackage{xy}
\usepackage{amssymb}
\usepackage{epsf,latexsym}
\usetikzlibrary{arrows,matrix,positioning}
\usetikzlibrary{cd}
\usepackage{mathrsfs}
\usepackage{epstopdf}
\usepackage{verbatim}
\usepackage{ctable}
\usepackage{wrapfig}
\usepackage{booktabs}
\usepackage{xspace}
\usepackage{url}
\usepackage{mathrsfs}
\usepackage{graphicx}
\usepackage{epstopdf}
\usepackage{algorithmic}
\usepackage{verbatim}
\usepackage{ctable}
\usepackage[margin=3.0cm]{geometry}
\usepackage{hyperref}

\ifpdf
\DeclareGraphicsExtensions{.eps,.pdf,.png,.jpg}
\else
\DeclareGraphicsExtensions{.eps}
\fi
\usepackage{enumitem}
\setlist[enumerate]{leftmargin=.5in}
\setlist[itemize]{leftmargin=.5in}

\makeatletter
\@namedef{subjclassname@2020}{
	\textup{2020} Mathematics Subject Classification}
\makeatother

\newcommand\wh[1]{\hstretch{2}{\hat{\hstretch{.5}{#1}}}}
\newcommand{\wDelta}{\wh{\Delta}}

\newcommand{\wsigma}{\wh{\sigma}}
\newcommand{\wtau}{\wh{\tau}}
\newcommand{\wgamma}{\wh{\gamma}}
\newcommand{\tS}{\mathtt{S}}
\newcommand{\wS}{\wh{\mathtt{S}}}

\newcommand{\tJ}{\mathtt{J}}
\newcommand{\tI}{\mathtt{I}}

\newcommand{\R}{{\ensuremath{\mathbb{R}}}}
\newcommand{\calR}{{\ensuremath{\mathcal{R}}}}
\newcommand{\calJ}{{\ensuremath{\mathcal{J}}}}

\newcommand{\im}{\ensuremath\mathrm{im\,}}

\newcommand{\LBcs}{\ensuremath{\mathrm{LB}^{\text{\faStar}}}} 
\newcommand{\LBos}{\ensuremath{\mathrm{LB}^{\text{\faStarO}}}}
\newcommand{\LB}{\ensuremath{\mathrm{LB}}} 
\newcommand{\spl}{\ensuremath{\mathcal{S}}}
\newcommand{\hspl}{\ensuremath{\mathcal{H}}}

\newcommand{\mesh}{tetrahedral partition}
\newcommand{\meshplural}{tetrahedral partitions}

\newtheorem{thm}{Theorem}[section]

\newtheorem{proposition}[thm]{Proposition}
\newtheorem{corollary}[thm]{Corollary}
\newtheorem{theorem}[thm]{Theorem}

\theoremstyle{definition}
\newtheorem{definition}[thm]{Definition}
\newtheorem{convention}[thm]{Convention}
\newtheorem{example}[thm]{Example}
\newtheorem{notation}[thm]{Notation}
\newtheorem{remark}[thm]{Remark}

\newtheorem{construction}[thm]{Construction}

\title{A lower bound for the dimension of tetrahedral splines in large degree}

\author[M. DiPasquale]{Michael DiPasquale}
\address{Michael DiPasquale\\     
	Department of Mathematics\\     
	Colorado State University}  
\email{michael.dipasquale@colostate.edu}
\urladdr{\url{https://midipasq.github.io/}}
\author[N. Villamizar]{Nelly Villamizar} 
\address{Nelly Villamizar\\
	Department of Mathematics\\
	Swansea University}
\email{n.y.villamizar@swansea.ac.uk}
\urladdr{\url{https://sites.google.com/site/nvillami}}

\begin{document}

\begin{abstract}
We derive a formula which is a lower bound on the dimension of trivariate splines on a tetrahedral partition which are continuously differentiable of order $r$ in large enough degree.  While this formula may fail to be a lower bound on the dimension of the spline space in low degree, we illustrate in several examples considered by Alfeld and Schumaker that our formula may give the exact dimension of the spline space in large enough degree if vertex positions are generic.  In contrast, for splines continuously differentiable of order $r>1$, every lower bound in the literature diverges (often significantly) in large degree from the dimension of the spline space in these examples.  We derive the bound using commutative and homological algebra.
\end{abstract}


\maketitle

\section{Introduction}
A multivariate spline is a piecewise polynomial function on a partition $\Delta$ of some domain $\Omega\subset\R^n$ which is continuously differentiable to order $r$ for some integer $r\ge 0$.  Multivariate splines play an important role in many areas such as finite elements, computer-aided design, isogeometric analysis, and data fitting~\cite{LaiSchumaker,Isogeometric}.  Splines on both triangulations and \meshplural{} have been used to solve boundary value problems by the finite element method; some early references are \cite{Courant,Strang,Z1973}, see also \cite{LaiSchumaker} and the references therein.  For quite recent applications in isogeometric analysis, in \cite{EE2016,EE2017}, Engvall and Evans outline frameworks to parametrize volumes for isogeometric analysis using triangular and tetrahedral B\'ezier elements.  While Engvall and Evans in~\cite{EE2017} focus on $C^0$ elements, $C^r$ tetrahedral B\'ezier elements are also used for isogeometric analysis -- see Xia and Qiang~\cite{XQ17}.  In these applications it is important to construct a basis, often with prescribed properties, for splines of bounded total degree.  Thus it is important to compute the dimension of the space of multivariate splines of bounded degree on a fixed partition.  We write $\spl^r_d(\Delta)$ for the vector space of piecewise polynomial functions of degree at most $d$ on the partition $\Delta$ which are continuously differentiable of order $r$.

A formula for the dimension of $C^1$ splines on triangulations was proposed by Strang~\cite{Strang} and proved for \textit{generic} triangulations by Billera~\cite{Homology}.  Subsequently the problem of computing the dimension of planar splines on triangulations has received considerable attention using a wide variety of techniques, see~\cite{SchumakerU,AS4r,AS3r,SuperSpline,WhiteleyComb,WhiteleyM,Homology,DimSeries,LCoho,MinReg}.  Alfeld and Schumaker show in~\cite{AS3r} that the dimension of $\spl^r_d(\Delta)$, for (most) planar triangulations $\Delta$ and $d\ge 3r+1$, is given by a quadratic polynomial in $d$ whose coefficients are determined from simple data of the triangulation.  The computation of $\dim \spl^r_d(\Delta)$ for planar $\Delta$ when $r+1\le d\le 3r$ remains an open problem, although Whiteley has shown that there are only trivial splines on $\Delta$ in degrees at most $\frac{3r+1}{2}$ if $\Delta$ is generic with a triangular boundary \cite{WhiteleyComb}.  (This result of Whiteley is an essential ingredient of our lower bound for trivariate splines.)

The literature on computing the dimension of trivariate splines on \meshplural{} is much less conclusive.  The dimension has been computed if $r=0$ (see~\cite{LocSup} or~\cite{Alg}), and also if $r=1$, $d\ge 8$, and $\Delta$ is generic by Alfeld, Schumaker, and Whiteley~\cite{ASWTet}.  For $r>1$ bounds on $\dim \spl^r_d(\Delta)$ have been computed in~\cite{Alfeld96,Lau,Tri,D3}. 
A major difficulty is that computing $\dim \spl^r_d(\Delta)$ exactly in large degree for arbitrary \meshplural{} cannot be done without computing the dimension of splines on planar triangulations exactly in \textit{all} degrees (see~\cite[Remark~65]{ASWTet}).  More precisely, to compute $\dim \spl^r_d(\Delta)$ exactly for $d\gg 0$, we must be able to compute the space of \textit{homogeneous} splines $\dim \hspl^r_d(\Delta_\gamma)$ exactly in all degrees, where $\gamma$ is a vertex of $\Delta$ and $\Delta_\gamma$ is the \textit{star} of $\gamma$  (that is, $\Delta_\gamma$ consists of all tetrahedra having $\gamma$ as a vertex).  The computation of such spline spaces has only been made for $r\le 1$; for $r=1$ $\Delta$ is required to be generic~\cite{ASWTet}.  For this crucial computation we rely on our previous paper~\cite{PaperA}, where we establish a lower bound on the dimension of homogeneous splines on vertex stars.

In our main result, Theorem~\ref{thm:LBGenericTet}, we establish a formula which is a lower bound on the dimension of the spline space on most \meshplural{} of interest (any triangulation of a compact three-manifold with boundary) in large enough degree.  While we have no proof of what degree is large enough, empirical evidence suggests that, for generic $\Delta$, our formula begins to be a lower bound in degrees close to the \textit{initial degree} of $\spl^r(\Delta)$; by the initial degree of $\spl^r(\Delta)$ we mean the smallest degree $d$ in which $\spl^r_d(\Delta)$ admits a spline which is not globally polynomial.  If $\Delta$ is generic, in Section~\ref{sec:Examples} we illustrate for several examples considered by Alfeld and Schumaker~\cite{Tri} that our formula gives the exact dimension of $\spl^r_d(\Delta)$ beginning at the initial degree of $\spl^r(\Delta)$.  It is worth noting that none of the lower bounds in the literature~\cite{Lau,Tri,D3} give the exact dimension of the generic spline space (even in large degree) on these examples if $r\ge 2$.

The paper is organized as follows.  In Section~\ref{s:BoundStatement} we explicitly state our lower bound in purely numerical terms allowing a straightforward application of the formula and illustrate in an example.  In Section~\ref{sec:Background} we set up notation and give relevant homological background, and in Section~\ref{sec:general_lowerbound} we prove the bound of Theorem~\ref{thm:LBGenericTet}.
Section~\ref{sec:Examples} is devoted to illustrating our bounds in a number of examples and comparing them to the bounds in~\cite{Tri,D3}.  Finally, we give some concluding remarks in Section~\ref{sec:ConcludingRemarks}.  We draw special attention to Remark~\ref{rem:AS8r+1}, as we think it likely that work of Alfeld, Schumaker, and Sirvent~\cite{LocSup} implies that our formula is a lower bound in degrees at least $8r+1$.  Our methods are sufficiently different from~\cite{LocSup} that we do not attempt to prove this here.

\section{The lower bound}~\label{s:BoundStatement}
Throughout we let $\Delta$ be a \mesh{}.  We are more precise in Section~\ref{sec:Background}; for now it is sufficient for the reader to think of a \mesh{} as a triangulation of a three-dimensional polytope.   
We use $\Delta_i$ and $\Delta^\circ_i$ to denote the $i$-faces and interior $i$-faces (respectively) of $\Delta$.  We put $f_i(\Delta)=|\Delta_i|$ and $f^\circ_i(\Delta)=|\Delta^\circ_i|$ (if $\Delta$ is clear we simply write $f_i$ and $f^\circ_i$).  We define the following data for each edge of $\Delta$.
\begin{notation}[Data attached to edges]\label{not:EdgeData}
	For a given $r\geq 0$ and $\tau\in\Delta_1$, 
	\begin{itemize}[leftmargin=*]
		\renewcommand{\labelitemi}{\scalebox{0.5}{$\blacksquare$}}
		\item let $t_\tau=\min\{n_\tau,r+2\}$, where  $n_\tau=\#\{\sigma\in\Delta_2:\tau\subset\sigma\}$ is the number of two-dimensional faces having $\tau$ as an edge;
		\item and the constants \ 
		$\displaystyle q_\tau = \biggl\lfloor \frac{t_\tau(r+1)}{t_\tau-1}\biggr\rfloor,\quad  a_\tau=t_\tau(r+1)-(t_\tau-1) q_\tau\ , \text{ and}\quad
		b_\tau=t_\tau-1-a_\tau\,.$
		($q_\tau$ and $a_\tau$ 
		are the quotient and remainder obtained when dividing  $t_\tau(r+1)$ by $t_\tau-1$\,.)
	\end{itemize}
\end{notation}
\noindent   
Given a vertex $\gamma\in\Delta$, we call the set of tetrahedra of $\Delta$ which contain $\gamma$ the \textit{star of $\gamma$} and we denote this \mesh{} by $\Delta_\gamma$.  If $\gamma$ is an interior vertex of $\Delta$, so $\gamma$ is completely surrounded by tetrahedra, then we call $\Delta_\gamma$ a \textit{closed} vertex star.  If $\gamma$ is a boundary vertex of $\Delta$, so $\gamma$ is not completely surrounded by tetrahedra, then we call $\Delta_\gamma$ an \textit{open} vertex star.  For a closed vertex star $\Delta_\gamma$ we define the constant
\begin{equation}\label{eq:Dgamma}
D_\gamma:=\left\lbrace
\begin{array}{ll}
2r & f^\circ_1=4\\
\lfloor (5r+2)/3 \rfloor & f_1^\circ=5\\
\lfloor (3r+1)/2 \rfloor & f_1^\circ\ge 6\ .
\end{array}
\right.
\end{equation}
\noindent The following convention for binomial coefficients is crucial in all our formulas.

\begin{convention}\label{conv:BinomialCoefficients}
	For binomial coefficients we put $\binom{n}{k}=0$ when $n<k$.
\end{convention}
\noindent If $\Delta_\gamma$ is a closed vertex star we define
\begin{align}\label{eq:LBclosedstar}
	\LBcs(\Delta_\gamma,d,r):= 2\binom{d+2}{2}+ & \left(f^\circ_2(\Delta_\gamma)-\sum\limits_{\tau\in\Delta^\circ_1} t_\tau \right) \binom{d+1-r}{2}\\
	&+\sum\limits_{\tau\in\Delta^\circ_1} \left( a_\tau\binom{d+1-q_\tau}{2}+b_\tau\binom{d+2-q_\tau}{2}\right).\nonumber
\end{align}
We write $\LBcs(d)$ instead of $\LBcs(\Delta,d,r)$ if $\Delta, r$ are understood.  In~\cite{PaperA} we show that $\LBcs(\Delta,d,r)$ is a lower bound for \textit{homogeneous} splines on a generic closed vertex star for $d>D_\gamma$ and~\cite{ANS96} shows there is equality for $d\ge 3r+2$.  If $\Delta_\gamma$ is an \textit{open} vertex star we define
\begin{align}\label{eq:LBopenstar}
	\LBos(\Delta_\gamma,d,r) := \binom{d+2}{2}+ & \left(f^\circ_2(\Delta_\gamma)-\sum\limits_{\tau\in\Delta^\circ_1} t_\tau \right) \binom{d+1-r}{2}\\
	&+\sum\limits_{\tau\in\Delta^\circ_1} \left( a_\tau\binom{d+1-q_\tau}{2}+b_\tau\binom{d+2-q_\tau}{2}\right). \nonumber
\end{align}
Again we write $\LBos(d)$ if $\Delta,r$ are understood.  In~\cite{ANS96} it is shown that $\LBos(\Delta,d,r)$ is a lower bound for homogeneous splines on an open vertex star, with equality if $d\ge 3r+2$.

If $\gamma$ is a vertex of $\Delta$ we attach the following constant to $\gamma$, which we call $N_\gamma$.  For a real number $r$, we put $[r]_+=\max\{r,0\}$.
\begin{equation}\label{eq:Ngamma}
N_\gamma=
\begin{cases}
\sum\limits_{d=r+1}^{D_\gamma} \left[\binom{d+2}{2}-\LBcs(d)\right]
+\sum\limits_{d=D_\gamma+1}^{3r+1} \left[ \binom{d+2}{2}-\LBcs(d) \right]_+ & \mbox{ if } \gamma\in\Delta^\circ_0
\\
\sum\limits_{d=r+1}^{3r+1} \left[\binom{d+2}{2}-\LBos(d)\right]_+ & \mbox{ if } \gamma\in\Delta_0\setminus \Delta^\circ_0 \ ,
\end{cases}
\end{equation}
where $D_\gamma$ is the constant attached to closed vertex stars defined in Equation~\eqref{eq:Dgamma}.

\begin{remark}
	When $\gamma\in\Delta^\circ_0$ and $r+1\le d\le D_\gamma$, notice that the contribution to $N_\gamma$ can be negative, while if $d>D_\gamma$, only positive contributions are counted.  This is a crucial difference between the contributions from interior vertices and the contributions from boundary vertices.
\end{remark}

\begin{theorem}[Lower bound in large degree for \meshplural{}]\label{thm:LBGenericTet}
	Suppose $\Delta$ is a \mesh{}.  If $d\gg 0$ then $\dim \spl^r_d(\Delta)\ge \LB(\Delta,d,r)$, where
	\begin{align}\label{eq:LBtetrahedral}
		\LB(\Delta,d,r) := & \bigl(f_3-f^\circ_2+f^\circ_1\bigr)\binom{d+3}{3}+ \left(f^\circ_2-\sum\limits_{\tau\in\Delta^\circ_1} t_\tau \right)\binom{d+2-r}{3}\\\nonumber
		&+\sum\limits_{\tau\in\Delta^\circ_1} \left( a_\tau\binom{d+2-q_\tau}{3}+b_\tau\binom{d+3-q_\tau}{3}\right)-f^\circ_0\binom{r+3}{3}+\sum_{\gamma\in\Delta_0} N_\gamma \ .
	\end{align}
\end{theorem}
\noindent If $\Delta$ and $r$ are understood then we abbreviate $\LB(\Delta,d,r)$ to $\LB(d)$.
\subsection{Example}\label{exam:intro}
We illustrate Theorem~\ref{thm:LBGenericTet} for $C^2$ splines on the \mesh{} in Fig.~\ref{fig:3DMS}, which is a three-dimensional analog of the Morgan-Scott triangulation~\cite{Morgan}.  
\begin{figure}
	\centering
	\includegraphics[scale=.732]{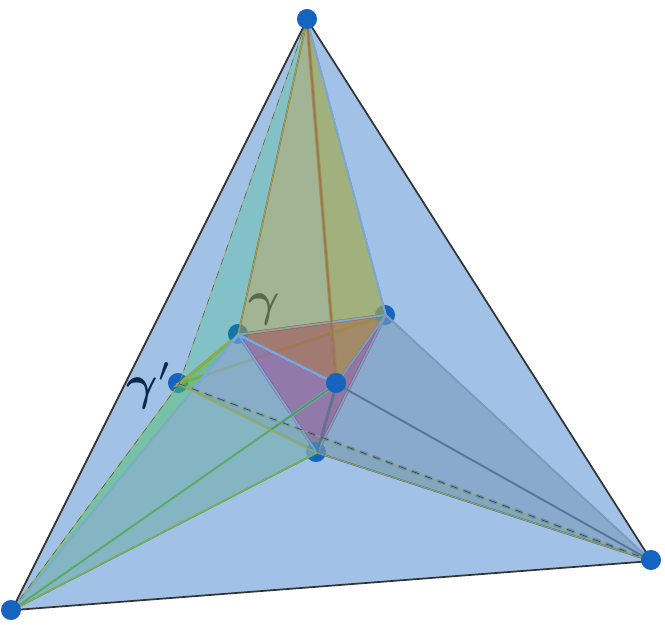}
	\includegraphics[scale=.732]{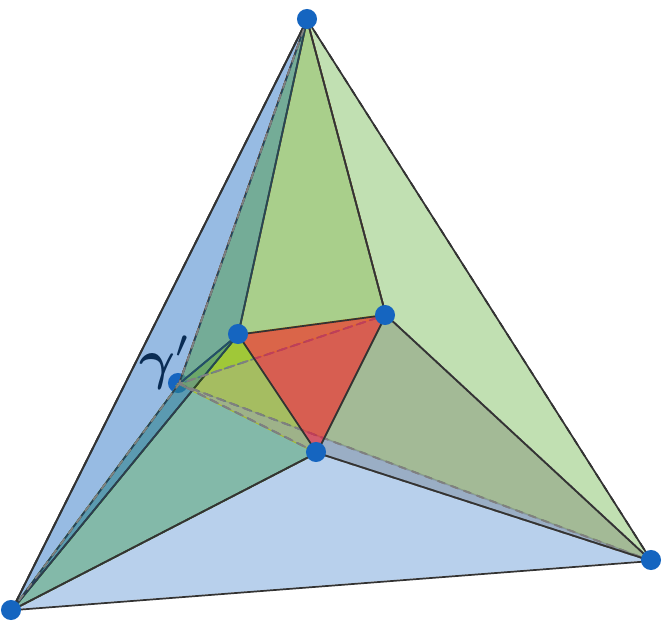}
	\includegraphics[scale=.732]{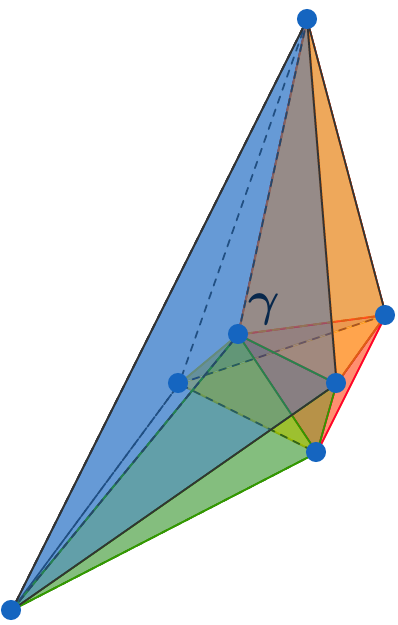}
	\caption{A three-dimensional version of the Morgan--Scott triangulation, the star of the boundary vertex $\gamma'$ (center), and the star of the interior vertex $\gamma$ (right).}\label{fig:3DMS}
\end{figure}
If $\gamma$ is an interior vertex then $\Delta_\gamma$ is the triangulated octahedron on the right in Fig.~\ref{fig:3DMS}.  
We have $f^\circ_0(\Delta_\gamma)=1$, $f^\circ_1(\Delta_\gamma)=6$, and $f^\circ_2(\Delta_\gamma)=12$.  For every $\tau\in(\Delta_\gamma)_1^\circ$, we have $n_\tau=4$ and hence $t_\tau=\min\bigl\{n_\tau,r+2\bigr\}=4$.  We compute $q_\tau=4,a_\tau=0,$ and $b_\tau=3$, hence by Equation~\eqref{eq:LBclosedstar},
\[
\LBcs\bigl(\Delta_\gamma,d,2\bigr)=2\binom{d+2}{2}-12\binom{d-1}{2}+18\binom{d-2}{2}\ .
\]
If $\gamma'$ is a boundary vertex, then $\Delta_{\gamma'}$ is the cone over the Morgan-Scott triangulation (see the star of vertex $\gamma'$ in Fig.~\ref{fig:3DMS}).  
We have $f^\circ_0(\Delta_{\gamma'})=0$, $f^\circ_1(\Delta_{\gamma'})=3$, and $f^\circ_2(\Delta_{\gamma'})=9$.  For every $\tau\in(\Delta_{\gamma'})_1^\circ$, we have $n_\tau=4$ and hence $t_\tau=\min\bigl\{n_\tau,r+2\bigr\}=4$.  
Again we have $q_\tau=4,a_\tau=0,$ and $b_\tau=3$.  Thus, following Equation~\eqref{eq:LBopenstar},
\[
\LBos\bigl(\Delta_{\gamma'},d,2\bigr)=\binom{d+2}{2}-3\binom{d-1}{2}+9\binom{d-2}{2}\ .
\]
In Table~\ref{tbl:genericMS} we record the values of $\LBcs\bigl(\Delta_\gamma,d,2\bigr)$, $\LBos\bigl(\Delta_{\gamma'},d,2\bigr)$, and $\binom{d+2}{2}$ where $\gamma$ is an interior vertex of $\Delta$ and $\gamma'$ is a boundary vertex of $\Delta$.

\begin{table}
	\centering
	\renewcommand{\arraystretch}{1.5}
	\begin{tabular}{c|cccccccc}
		$d$ & 3 & 4 & 5 & 6 & 7 & 8 & 9 & 10 \\
		\hline
		$\binom{d+2}{2}$ & 10 & 15 & 21 & 28 & 36 & 45 & 55 & 66 \\
		\hline
		$\LBcs(\Delta_\gamma,d,2)$      & 8 & 12 & 24 & 44 & 72 & 108 & 152 & 204 \\
		\hline
		$\LBos(\Delta_{\gamma'},d,2)$     & 7 & 15 & 30 & 52 & 81 & 117 & 160 & 210 \\
	\end{tabular}
	\caption{Lower bounds for the star of an interior ($\gamma$) and boundary ($\gamma'$) vertex of the simplicial complex in Fig.~\ref{fig:3DMS}. \vspace{-0.5cm}}\label{tbl:genericMS}
\end{table}

Now we turn to computing the bound $\LB\bigl(\Delta,d,2\bigr)$ in Theorem~\ref{thm:LBGenericTet} for $\dim \spl^2_d(\Delta)$, where $\Delta$ is the full simplicial complex depicted in Fig.~\ref{fig:3DMS}.  If $\gamma$ is a boundary vertex then $N_\gamma=3$ (corresponding to the one difference in degree $3$ in Table~\ref{tbl:genericMS}).  If $\gamma$ is an interior vertex then $D_\gamma=3$.  Reading down each column in the first two rows of Table~\ref{tbl:genericMS} we get $N_\gamma=(10-8)+(15-12)=5$.  Thus $\sum_{\gamma\in\Delta_0} N_\gamma=4\cdot 3+4\cdot(5)=32$.

For the remaining statistics we have $f_0^\circ=4,f_1^\circ=18,f_2^\circ=28,$ and $f_3=15$.  For each interior $1$-face $\tau$ we have $n_\tau=t_\tau=4, d_\tau=4, a_\tau=0, b_\tau=3$.  Thus, by Theorem~\ref{thm:LBGenericTet}, 
\[
\LB(\Delta,d,2)= 5\binom{d+3}{3}-44\binom{d}{3}+54\binom{d-1}{3}-8 = \frac{5}{2}d^3-27d^2+\frac{187}{2}d-57,
\]
where the second equality holds as long as $d\ge 1$.  Table~\ref{tbl:3DMSLower} compares the values of $\LB(\Delta,d,2)$ and $\dim \spl^2_d(\Delta)$ for generic positions of the vertices of $\Delta$.  Notice that while $\LB(\Delta,d,2)$ is neither an upper or lower bound for $d\le 6$, it predicts the correct dimension of the generic spline space for $d\ge 7$.  Incidentally, $d=7$ is the initial degree of $\spl^2(\Delta)$; that is, the first non-trivial splines appear in degree $7$.  We computed the exact dimension of the spline space for generic vertex positions using the \texttt{Algebraic Splines} package in Macaulay2~\cite{M2}.  Furthermore, a computation in Macaulay2 shows that $\dim \spl^r_d(\Delta)=\frac{5}{2}d^3-27d^2+\frac{187}{2}d-57$ for $d\gg 0$, so our lower bound gives the exact dimension of the spline space for $r=2$ when $d\ge 7$.  Code to compute all examples in this paper can be found on the first author's website under the Research tab:~\url{https://midipasq.github.io/}.

\begin{table}
	\centering
	\renewcommand{\arraystretch}{1.5}
	\begin{tabular}{c|ccccccccccc}
		$d$ & 0 & 1 & 2 & 3 & 4 & 5 & 6 & 7 & 8 & 9 & 10 \\
		\hline
		$\binom{d+3}{3}$ & 1 & 4 & 10 & 20 & 35 & 56 & 84 & 120 & 165 & 220 & 286 \\
		\hline
		$\LB(\Delta,d,2)$ & -57 & 12 & 42 & 48 & 45 & 48 & 72 & 132 & 243 & 420 & 678 \\
		\hline
		$\dim \spl^2_d(\Delta)$ & 1 & 4 & 10 & 20 & 35 & 56 & 84 & 132 & 243 & 420 & 678
	\end{tabular}
	\caption{Illustrating Theorem~\ref{thm:LBGenericTet} for the \mesh{} in Fig.~\ref{fig:3DMS}. \vspace{-0.5cm}}\label{tbl:3DMSLower}
\end{table}

\section{Background and Homological Methods}\label{sec:Background}
In this section we introduce the homological methods of Billera~\cite{Homology} and Schenck and Stillman~\cite{LCoho}.  A \textit{simplex} in $\R^n$ is the convex hull of $i\le n+1$ vertices which are in linearly general position (no three on a line, no four on a plane, etc.).  A \textit{face} of a simplex is the convex hull of any subset of the vertices which define it (thus a face of a simplex is a simplex).  An $i$-simplex (or $i$-face) is the convex hull of $i+1$ vertices in linearly general position; $i$ is the \textit{dimension} of the $i$-simplex or $i$-face.

\begin{definition}
	A \textit{simplicial complex} $\Delta$ is a collection of simplices in $\R^n$ satisfying:
	\begin{itemize}[leftmargin=*]
		\renewcommand{\labelitemi}{\scalebox{0.5}{$\blacksquare$}}	
		\item If $\beta\in\Delta$ then so are all of its faces.
		\item If $\beta_1,\beta_2\in\Delta$ then $\beta_1\cap\beta_2$ is either empty or a proper face of both $\beta_1$ and $\beta_2$.
	\end{itemize}
	We also refer to the simplices of $\Delta$ as \textit{faces} of $\Delta$.  The \textit{dimension} of $\Delta$ is the dimension of a maximal simplex of $\Delta$ under inclusion. If all maximal simplices have the same dimension we said that $\Delta$ is \textit{pure} .
\end{definition}
In this paper we only consider \textit{finite} simplicial complexes.  If $\beta$ is a face of $\Delta$ of dimension $i$ we call $\beta$ an $i$-face.  Denote by $\Delta_i$ and $\Delta_i^\circ$ the set $i$-faces of $\Delta$ and interior $i$-faces of $\Delta$, respectively.  We write $f_i(\Delta)$ and $f^\circ_i(\Delta)$ for the number of $i$-faces and interior $i$-faces, respectively (we write $f_i$ and $f^\circ_i$ if $\Delta$ is understood).  By an abuse of notation, we will identify $\Delta$ with its underlying space $\bigcup_{\beta\in\Delta} \beta\subset\R^n$.

\begin{definition}
	If $\Delta$ is a simplicial complex and $\beta$ is a face of $\Delta$, then the \textit{link} of $\beta$ is the set of all simplices $\gamma$ in $\Delta$ so that $\beta\cap\gamma=\emptyset$ and $\beta\cup\gamma$ is a face of $\Delta$.  The \textit{star} of $\beta$ is the union of the link of $\beta$ with the set of all simplices which contain $\beta$ (including $\beta$).  We denote the star of $\beta$ by $\Delta_\beta$.
\end{definition}
If $\gamma$ is a vertex of a simplicial complex $\Delta$ so that all maximal simplices of $\Delta$ contain $\gamma$ (so $\Delta_\gamma=\Delta$), then we call $\Delta$ the \textit{star of} $\gamma$ and we say $\Delta$ is a \textit{vertex star}.  If $\gamma$ is an \textit{interior} vertex we call $\Delta$ a closed vertex star and if $\gamma$ is a boundary vertex then we call $\Delta$ an open vertex star.

We refer to the set of points in $\R^{n+1}$ of unit norm as the $n$-sphere, and the set of points in $\R^n$ with norm at most one as the $n$-disk.  A \textit{homeomorphism} $f:X\to Y$ between two sets is a continuous bijection; if such an $f$ exists we say $X$ and $Y$ are \textit{homeomorphic}.

\begin{definition}[Simplicial $n$-manifold with boundary]
	If $\Delta$ is a finite simplicial complex in $\R^n$, we say it is a \textit{simplicial} $n$-\textit{manifold} with boundary if it satisfies the conditions:
	\begin{itemize}[leftmargin=*]
		\renewcommand{\labelitemi}{\scalebox{0.5}{$\blacksquare$}}	
		\item $\Delta$ is pure $n$-dimensional,
		\item the link of every vertex of $\Delta$ is homeomorphic to an $(n-1)$-sphere (if the vertex is \textit{interior}) or an $(n-1)$-disk (if the vertex is on the boundary),
		\item and every $(n-1)$-simplex of $\Delta$ is either the intersection of two $n$-simplices of $\Delta$ or it is on the boundary of $\Delta$ and so contained in only one $n$-simplex of $\Delta$.
	\end{itemize}
\end{definition}

\begin{example}
	Consider the simplicial complex in Fig.~\ref{fig:3DMS}, which is a simplicial $3$-mani\-fold with boundary homeomorphic to the $3$-disk.  The star of the interior vertex $\gamma$ is shown in the center of Fig.~\ref{fig:3DMS}; the link of the vertex $\gamma$ is obtained from the star of $\gamma$ by removing $\gamma$ and all simplices which contain it.  The link of $\gamma$ is homeomorphic to a $2$-sphere.  Likewise, the star of the boundary vertex $\gamma'$ is shown on the right in Fig.~\ref{fig:3DMS}; the link of the vertex $\gamma'$ is obtained from it by removing the vertex $\gamma'$ and all simplices which contain it.  The link of $\gamma'$ is the usual planar Morgan-Scott configuration~\cite{Morgan}, and is homeomorphic to a $2$-disk.
\end{example}

Throughout this paper we abuse notation by referring to a simplicial $n$-manifold with boundary simply as a simplicial complex.  We refer to a simplicial $2$-manifold with boundary as a \textit{triangulation} and a simplicial $3$-manifold with boundary as a \textit{\mesh{}}.

Write $\tS=\R[x_1,\ldots,x_n]$ for the polynomial ring in $n$ variables and $\tS_{\le d}$ for the $\R$-vector space of polynomials of total degree most $d$, and $\tS_d$ for the $\R$-vector space of polynomials which are homogeneous of degree exactly $d$.  For a fixed integer $r$, we denote by $C^r(\Delta)$ the set of all functions $F: \Delta\to\R$ which are continuously differentiable of order $r$.
\begin{definition}\label{def:SplineSpaces}
	Let $\Delta\subset\R^n$ be an $n$-dimensional simplicial complex.  We denote by
	\[
	\spl^r(\Delta):= \bigl\{F\in C^r(\Delta)\colon F|_\iota\in \tS\; \mbox{ for all }\iota\in\Delta_n \bigr\}
	\]
	the vector space of splines which are continuously differentiable of order $r$, by
	\[
	\spl^r_d(\Delta):= \bigl\{F\in C^r(\Delta)\colon F|_\iota\in \tS_{\le d}\; \mbox{ for all }\iota\in\Delta_n \bigr\}
	\]
	the subspace of $\spl^r(\Delta)$ consisting of splines of degree at most $d$, and by
	\[
	\hspl^r_d(\Delta):= \bigl\{F\in C^r(\Delta)\colon F|_\iota\in \tS_d\; \mbox{ for all }\iota\in\Delta_n\bigr\}
	\]
	the subspace of $\spl^r(\Delta)$ consisting of splines whose restriction to each $n$-dimensional simplex is a homogeneous polynomial of degree $d$.  We call splines in $\hspl^r_d(\Delta)$ \textit{homogeneous} splines.
\end{definition}
If $\Delta$ is the star of a vertex, then one can show that
\begin{equation}\label{eq:StarVertexHomogDecomp}
\spl^r(\Delta)\cong \bigoplus\limits_{i\ge 0} \hspl^r_i(\Delta), \mbox{ and}\quad \spl^r_d(\Delta)\cong \bigoplus\limits_{i=0}^d \hspl^r_i(\Delta),
\end{equation}
where the isomorphism is as $\R$-vector spaces.  We refer to the first isomorphism in~\eqref{eq:StarVertexHomogDecomp} as the \textit{graded structure} of $\spl^r(\Delta)$.  If $\Delta$ is not the star of a vertex, then~\eqref{eq:StarVertexHomogDecomp} does not hold for $\spl^r(\Delta)$; we summarize a coning construction of Billera and Rose under which~\eqref{eq:StarVertexHomogDecomp} will still be valid.

\begin{construction}\label{cons:coning}
	Let $\R^n$ have coordinates $x_1,\ldots,x_n$, $\R^{n+1}$ have coordinates $x_0,\ldots,x_n$, and define $\phi:\R^n\to\R^{n+1}$ by $\phi(x_1,\ldots,x_n)=(1,x_1,\ldots,x_n)$.  If $\sigma$ is a simplex in $\R^n$, the \textit{cone over} $\sigma$, denoted $\wh{\sigma}$, is the simplex in $\R^{n+1}$ which is the convex hull of the origin in $\R^{n+1}$ and $\phi(\sigma)$.  If $\Delta\subset\R^n$ is a simplicial complex, the \textit{cone over} $\Delta$, denoted $\wDelta$, is the simplicial complex consisting of the simplices $\bigl\{\wh{\beta}:\beta\in\Delta\bigr\}$ along with the origin in $\R^{n+1}$, which is called the \textit{cone vertex}.  We denote the polynomial ring $\R[x_0,x_1,\ldots,x_n]$ associated to $\wDelta$ by $\wS$.
\end{construction}
For any simplicial complex $\Delta\subset\R^n$, the simplicial complex $\wDelta\subset\R^{n+1}$ is an (open) vertex star of the cone vertex.  Thus~\eqref{eq:StarVertexHomogDecomp} yields
$
\spl^r(\wDelta)\cong \bigoplus\limits_{i\ge 0} \hspl^r_i(\wDelta)  \mbox{ and } \spl^r_d(\wDelta)\cong \bigoplus\limits_{i=0}^d \hspl^r_i(\wDelta).
$
Moreover, Billera and Rose show that
\begin{theorem}{\cite[Theorem~2.6]{DimSeries}}\label{thm:Homogenize}
	$\spl^r_d(\Delta)\cong \hspl^r_d(\wDelta)$.
\end{theorem}
Thus the study of spline spaces reduces to the study of \textit{homogeneous} spline spaces.

\begin{definition}\label{def:IdealGenPowersOfLinearForms}
	A subset $\tI\subset \tS$ is called an \textit{ideal} if, for every $f,g\in \tI$ and $h\in \tS$, $f+g\in\tI$ and $hf\in\tI$.  If $f_1,\ldots,f_k\in \tS$ are polynomials, we write $\langle f_i \rangle$ for the vector space of all polynomial multiples of $f_i$ ($i=1,\ldots,k$) and $\langle f_1,\ldots,f_k\rangle:=\sum_{i=1}^k \langle f_i\rangle$.  This is called the \textit{ideal generated by} $f_1,\ldots,f_k$.  We typically only use its vector space structure.
\end{definition}

\begin{definition}\label{def:IdealsOfFaces}
	Suppose $\Delta\subset\R^n$ is an $n$-dimensional simplicial complex. 
	If $\beta\in\Delta_n$ we define $\tJ(\beta)=0$. 
	If $\sigma\in\Delta_{n-1}$, let $\ell_\sigma$ be a choice of linear form vanishing on $\sigma$.  We define $\tJ(\sigma)=\langle \ell_\sigma^{r+1}\rangle$.  For any face $\beta\in\Delta_i$ where $i<n$ we define
	\[
	\tJ(\beta):=\sum_{\sigma\supseteq\beta} \tJ(\sigma)=\langle \ell^{r+1}_\sigma: \beta\subseteq \sigma\rangle.
	\]
\end{definition}
Billera and Rose show that if $\Delta$ is \textit{hereditary} (a hypothesis which is implied by ours) then

\begin{proposition}{\cite[Proposition~1.2]{DimSeries}}\label{prop:AlgebraicCriterion}
	$F\in\spl^r(\Delta)$ if and only if 
	\[
	F|_{\iota}-F|_{\iota'}\in \tJ(\sigma) \mbox{ \ for every }\iota, \iota'\in \Delta_n \mbox{\  satisfying } \iota\cap \iota'=\sigma\in\Delta_{n-1}.
	\]
\end{proposition}

\subsection{Chain Complexes}\label{ss:ChainComplexes}
If $C_0,\ldots,C_k$ are vector spaces and $\partial_i:C_i\to C_{i-1}$ ($i=1,\ldots,k$) are linear maps satisfying $\partial_{i-1}\circ\partial_i=0$ (for $i=2,\ldots,k$), then the collection of this data is called a \textit{chain complex}; this is typically recorded as
\[
\mathcal{C}: 0\rightarrow C_k \xrightarrow{\partial_k} C_{k-1}\xrightarrow{\partial_{n-1}}\cdots\xrightarrow{\partial_1} C_0\rightarrow 0.
\]
We call the subscript $i$ of $C_i$ the \textit{homological index} and refer to $C_i$ as the vector space of $\mathcal{C}$ in homological index $i$.  The \textit{homologies} of the chain complex are the quotient vector spaces $H_i(\mathcal{C})=\ker(\partial_{i-1})/\im(\partial_i)$ for $i=0,\ldots,k$. (We put $H_0(\mathcal{C})=C_0/\im(\partial_1)$ and $H_k(\mathcal{C})=\ker(\partial_k)$.)  Often $H_*(\mathcal{C})$ is used to denote the entire set of homology groups $H_0(\mathcal{C}),\ldots,H_k(\mathcal{C})$.  We are primarily concerned with a topological construction of chain complexes; see~\cite[Chapter~2]{Hatcher} for a standard reference.

We now define the chain complex introduced by Billera~\cite{Homology} and refined by Schenck and Stillman~\cite{LCoho}.
Let $\tS^{\Delta_i}$ ($i=0,\ldots,n$) denote the direct sum $\bigoplus_{\beta\in \Delta_i} \tS[\beta]$, where $[\beta]$ is a formal basis symbol corresponding to the $i$-face $\beta$.  Fix an ordering $\gamma_1,\ldots,\gamma_{f_0}$ of the vertices of $\Delta$.  Each $i$-face $\beta\in\Delta_i$ can be represented as an \textit{ordered} list $\beta=(\gamma_{j_0},\ldots,\gamma_{j_i})$ of $i+1$ vertices.  We define the \textit{simplicial boundary map} $\partial_i$ (for $i=1,2,3$) on the formal symbol $[\beta]=[\gamma_{j_0},\ldots,\gamma_{j_i}]$ by
$ 
\partial_i\bigl([\beta]\bigr)=\partial_i\bigl([\gamma_{j_0},\ldots,\gamma_{j_i}]\bigr)=\sum_{k=0}^i (-1)^i\bigl[\gamma_0,\ldots,\hat{\gamma}_{j_k},\ldots,\gamma_{j_i}\bigr]\ ,$
where $\hat{\gamma}_{j_k}$ means that the vertex $\gamma_{j_k}$ is omitted from the list.  We extend this map linearly to $\bigoplus_{\beta\in\Delta_i} \tS[\beta]$.  

It is straightforward to verify that $\partial_{i-1}\circ\partial_i=0$ for $i=2,\ldots,n$ (this only needs to be checked on the basis symbols $[\beta]$).  Clearly the simplicial boundary map $\partial_i$ can be restricted to a map $\partial_i: \tS^{\Delta^\circ_i}\to \tS^{\Delta^\circ_{i-1}}$ where all formal symbols corresponding to faces on the boundary of $\Delta$ are dropped.  We denote by $\calR[\Delta]$ the chain complex
\[
\calR[\Delta] \colon\quad  0\longrightarrow \tS^{\Delta_n}\xrightarrow{\partial_n} \tS^{\Delta^\circ_{n-1}} \xrightarrow{\partial_{n-1}} \cdots\xrightarrow{\partial_2} \tS^{\Delta^\circ_1} \xrightarrow{\partial_1} \tS^{\Delta^\circ_0} \longrightarrow 0\ .
\]
(This is the simplicial chain complex of $\Delta$ relative to its boundary $\partial\Delta$ with coefficients in $\tS$ -- see~\cite[Chapter~2.1]{Hatcher}).

We now put the vector spaces $\tJ(\beta)$ together to make a sub-chain complex of $\calR[\Delta]$
\[
\calJ[\Delta]\colon \quad  0\longrightarrow \bigoplus_{\iota\in\Delta_n} \tJ(\iota)=0 \rightarrow \bigoplus_{\sigma\in\Delta^\circ_{n-1}} \tJ(\sigma)\xrightarrow{\partial_{n-1}}\cdots\xrightarrow{\partial_2} \bigoplus_{\tau\in\Delta^\circ_1} \tJ(\tau) \xrightarrow{\partial_1} \bigoplus_{\gamma\in\Delta^\circ_0} \tJ(\gamma) \longrightarrow 0\ .
\]
The \textit{Billera-Schenck-Stillman chain complex} is the quotient of $\calR[\Delta]$ by $\calJ[\Delta]$, namely
\[
\calR/\calJ[\Delta] \colon \quad 0\longrightarrow \bigoplus_{\iota\in\Delta_{n}}  \tS\xrightarrow{\overline{\partial}_n} \bigoplus_{\sigma\in\Delta_{n-1}^\circ} \frac{\tS}{\tJ(\sigma)}\xrightarrow{\overline{\partial}_{n-1}} \cdots\xrightarrow{\overline{\partial}_2} \bigoplus_{\tau\in\Delta_1^\circ}\frac{\tS}{\tJ(\tau)}\xrightarrow{\overline{\partial}_1} \bigoplus_{\gamma\in\Delta_0^\circ}\frac{\tS}{\tJ(\gamma)}\longrightarrow 0\ .
\]
\begin{remark}
	If the simplicial complex $\Delta$ is fixed, we simply write $\calJ,\calR,$ and $\calR/\calJ$ for the chain complexes $\calJ[\Delta],\calR[\Delta]$, and $\calR/\calJ[\Delta]$, respectively.
\end{remark}

\begin{notation}\label{not:ConingAbuse}	
	We introduce a natural abuse of notation regarding the coning construction~\ref{cons:coning}.  If $\Delta$ is a simplicial complex and $\wDelta$ is the cone over $\Delta$, then $\wDelta$ is an open vertex star.  Hence there is no interior vertex of $\wDelta$ and thus the vector space of homological index $0$ in $\calJ[\wDelta],\calR[\wDelta],$ and $\calR/\calJ[\wDelta]$ is just zero.  We thus decrease the homological index by one of each of the vector spaces in $\calJ[\wDelta],\calR[\wDelta],$ and $\calR/\calJ[\wDelta]$.  Hence if $\Delta\subset\R^n$ and thus $\wDelta\subset\R^{n+1}$, $H_n(\calR/\calJ[\wDelta])$ is the top homology of the chain complex $\calR/\calJ[\wDelta]$, not $H_{n+1}(\calR/\calJ[\wDelta])$ (and likewise for lower indices).  Thus the vector space in homological index $i$ ($0\le i\le n$) in $\calR/\calJ[\wDelta]$ corresponds to the homological index $i$ in $\calR/\calJ[\Delta]$, so its summands are indexed by $\Delta^\circ_i$.
\end{notation}

The crucial observation of Billera is that $H_n(\calR/\calJ[\Delta])\cong \spl^r(\Delta)$; this follows from the criterion of Proposition~\ref{prop:AlgebraicCriterion}. 

\subsection{Graded structure}\label{ss:GradedStructure}
The vector space $\tJ(\beta)$ is infinite-dimensional for each face $\beta\in\Delta$ which is not a tetrahedron.  Thus the constituents of the chain complexes $\calJ[\Delta],\calR[\Delta],$ and $\calR/\calJ[\Delta]$ are also infinite-dimensional.  In order to get a chain complex of finite dimensional vector spaces to relate to the fundamental spaces of interest ($\spl^r_d(\Delta)$ and $\hspl^r_d(\Delta)$), we make use of a \textit{graded structure}.

\begin{definition}\label{def:GradedStructure}
	Let $V$ be a real vector space and suppose $V_i$ is a finite-dimensional vector subspace of $V$ for every integer $i\ge 0$.  If $V\cong \bigoplus_{i\ge 0} V_i$, then we refer to this isomorphism as a \textit{graded structure} of $V$ and we call $V$ a \textit{graded} vector space.  In particular, if $\tJ\subset \tS$ is an ideal (c.f. Definition~\ref{def:IdealGenPowersOfLinearForms}), then we write $\tJ_d$ for the vector space of homogeneous polynomials of degree $d$ in $\tJ$.  If $\tJ\cong\bigoplus_{d\ge 0} \tJ_d$ then we call $\tJ$ a \textit{graded ideal} of $\tS$.
\end{definition}

\begin{definition}\label{def:GradedChainComplex}
	If $\mathcal{C}: 0\rightarrow C_n\xrightarrow{\partial_n}\cdots\xrightarrow{\partial_1} C_0 \rightarrow 0$ is a chain complex of vector spaces so that
	\begin{enumerate}
		\item The vector space $C_j$ has a graded structure $C_j\cong \bigoplus_{i\ge 0} (C_j)_i$ for $j=0,\ldots,n$ and
		\item The map $\partial_j:C_j\to C_{j-1}$ satisfies $\partial_j((C_j)_i)\subset (C_{j-1})_i$ for $j=1,\ldots,n$\ ,
	\end{enumerate}
	then $\mathcal{C}_d:=0\rightarrow (C_n)_d\xrightarrow{\partial_n} (C_{n-1})_d \xrightarrow{\partial_{n-1}}\cdots \xrightarrow{\partial_1} (C_0)_d \rightarrow 0$ is a chain complex which we call the \textit{degree $d$ strand} of $\mathcal{C}$.
	In this case we say $\mathcal{C}$ is \textit{graded} with \textit{graded structure} $\mathcal{C}\cong \bigoplus_{d\ge 0} \mathcal{C}_d$.
	
	If a chain complex $\mathcal{C}$ has a graded structure $\mathcal{C}\cong \bigoplus_{d\ge 0}\mathcal{C}_d$, it is straightforward to see that the homologies of $\mathcal{C}$ also have the graded structure $H_i(\mathcal{C})\cong \bigoplus_{d\ge 0} H_i(\mathcal{C})_d$, where $H_i(\mathcal{C})_d:= H_i(\mathcal{C}_d)$ is the $i$th homology of the degree $d$ strand.
\end{definition}

\begin{remark}
	The isomorphisms~\eqref{eq:StarVertexHomogDecomp} show that $\spl^r(\Delta)$ has a graded structure if $\Delta$ is the star of a vertex.
\end{remark}

If $\Delta$ is a vertex star of $\gamma$ (assumed to be the origin) and $\gamma\in\beta$, then the linear forms whose powers generate $\tJ(\beta)$ have no constant term and $\tJ(\beta)$ is a graded ideal.  It is straightforward to see that the simplicial boundary map respects this graded structure (i.e. property (2) of Definition~\ref{def:GradedChainComplex} is satisfied), so if $\Delta$ is a vertex star then the chain complexes $\calJ[\Delta],\calR[\Delta],$ and $\calR/\calJ[\Delta]$ also have a graded structure, along with their homologies.  In particular, $\hspl^r_d(\Delta)\cong H_n\bigl(\calR/\calJ[\Delta]\bigr)_d$ if $\Delta\subset\R^n$ is a vertex star.  If $\Delta$ is not necessarily a vertex star, we can take advantage of the coning construction $\Delta\to\wh{\Delta}$ to obtain a graded structure.  Keeping in mind Theorem~\ref{thm:Homogenize} and Notation~\ref{not:ConingAbuse}, we have $\spl^r_d(\Delta)\cong \hspl^r_d(\wDelta)\cong \ker(\overline{\partial}_n)_d\cong H_n\bigl(\calR/\calJ[\wDelta]\bigr)_d$.

\subsection{Euler characteristic and dimension formulas}
If $\mathcal{C}\colon 0\rightarrow C_n\rightarrow C_{n-1} \rightarrow \cdots \rightarrow C_0\rightarrow 0$ is a chain complex with a graded structure, we write $\chi(\mathcal{C},d)=\sum_{i=0}^n (-1)^{n-i}\dim (C_i)_d$.  This is the \textit{Euler-Poincar\'e characteristic} of $\mathcal{C}_d$.  The rank-nullity theorem yields:
\begin{equation}\label{eq:EulerCharacteristic}
\chi(\mathcal{C},d)=\sum_{i=0}^n (-1)^{n-i} \dim H_i(\mathcal{C})_d.
\end{equation}
The three chain complexes $\calJ,\calR,$ and $\calR/\calJ$ fit into the short exact sequence of chain complexes $0\rightarrow \calJ\rightarrow \calR\rightarrow \calR/\calJ \rightarrow 0$.  Correspondingly there is the long exact sequence:
\[
0\rightarrow H_n(\calJ)\rightarrow H_n(\calR) \rightarrow \cdots \rightarrow H_1(\calR/\calJ)\rightarrow H_0(\calJ)\rightarrow H_0(\calR)\rightarrow H_0(\calR/\calJ)\rightarrow 0.
\]
The short exact sequence $0\rightarrow \calJ\rightarrow \calR\rightarrow \calR/\calJ\rightarrow 0$ also yields
\begin{equation}\label{eq:EulerCharESCC}
\chi(\calR/\calJ,d)=\chi(\calR,d)+\chi(\calJ,d)\ .
\end{equation}
There is a sum instead of a difference on the right hand side of Equation~\eqref{eq:EulerCharESCC} because the first non-zero term in the chain complex $\calJ$ has homological degree $n-1$ instead of $n$.

\begin{proposition}\label{prop:FrequentlyUsedIsomorphisms}
	For an $n$-dimensional simplicial complex $\Delta$ in $\R^n$, $H_n(\calR/\calJ[\wDelta])_d\cong \spl^r_d(\Delta)$ and $H_0(\calR/\calJ[\Delta])=0$.  If $\Delta$ is a vertex star, $H_n(\calR/\calJ[\Delta])_d\cong \hspl^r_d(\Delta)$.  If $\Delta$ is connected, then $H_0(\calR/\calJ[\Delta])=0$.  If $\Delta$ is a vertex star whose link is homeomorphic to an $(n-1)$-sphere or an $(n-1)$-disk, then $\spl^r(\Delta)\cong H_n(\calR/\calJ)\cong \tS\oplus H_{n-1}(\calJ)$ and $H_i(\calR/\calJ)\cong H_{i-1}(\calJ)$ for $i=1,\ldots,n-1$.
\end{proposition}
\begin{proof}
	By Theorem~\ref{thm:Homogenize} and Proposition~\ref{prop:AlgebraicCriterion}, $\spl^r_d(\Delta) \cong \hspl^r_d(\wDelta) \cong H_n(\calR/\calJ[\Delta])_d$.  Since every vertex can be connected to the boundary of $\Delta$ by a path consisting of interior edges, $\partial_1:\tS^{f^\circ_1}\to\tS^{f^\circ_0}$ is surjective and thus $H_0(\calR[\Delta])=0$, hence $H_0(\calR/\calJ[\Delta])=0$ by the long exact sequence associated to $0\to \calJ\to\calR\to \calR/\calJ\to 0$\ .
	
	The hypothesis that $\Delta$ is a vertex star whose link is homeomorphic to an $(n-1)$-sphere or an $(n-1)$-disk implies that $H_i(\calR[\Delta])=0$ for $0\le i<n$ and $H_n(\calR[\Delta])\cong \tS$ (by excision~\cite[Proposition~2.22]{Hatcher}, the homology of $\Delta$ relative to its boundary coincides with the homology of the $n$-sphere, which gives the claimed homologies).  Then the last result follows from the long exact sequence associated to $0\rightarrow\calJ\rightarrow\calR\rightarrow\calR/\calJ\rightarrow 0$.
\end{proof}

\begin{remark}
	If $\Delta$ is homeomorphic to an $n$-disk, then the copy of $\tS$ in $\spl^r(\Delta)\cong \tS\oplus H_{n-1}(\calJ)$ corresponds to the globally polynomial splines, while the the so-called \textit{smoothing cofactors} are encoded by the map
	\[
	\bigoplus_{\sigma\in\Delta^\circ_{n-1}} \tJ(\sigma) \xrightarrow{\partial_{n-1}} \bigoplus_{\tau\in\Delta^\circ_{n-2}} \tJ(\beta)\ .
	\]
\end{remark}
\begin{proposition}\label{prop:EulerCharacteristicAndDimension}
	If $\Delta$ is a \mesh{} then
	\begin{multline}\label{eq:dimformula}
		\dim \spl^r_d(\Delta)=\bigl(f_3-f^\circ_2+f^\circ_1-f^\circ_0\bigr)\dim\wS_d+\chi\bigl(\calJ[\wDelta],d\bigr)\\ +\dim H_2\bigl(\calR/\calJ[\wDelta]\bigr)_d- \dim H_1\bigl(\calR/\calJ[\wDelta]\bigr)_d.
	\end{multline}
	If $\Delta$ is a tetrahedral vertex star whose link is homeomorphic to a $2$-sphere or a $2$-disk then
	\begin{equation}\label{eq:celldimformula}
	\dim \hspl^r_d(\Delta)=\dim \tS_d+\chi(\calJ[\Delta],d)+\dim H_1(\calJ[\Delta])_d.
	\end{equation}
\end{proposition}
\begin{proof}
	First we make use of the identifications $\spl^r_d(\Delta)\cong \hspl^r_d(\wDelta)$ and $H_3(\calR/\calJ[\wDelta])_d\cong \hspl^r_d(\wDelta)$ of Theorems~\ref{thm:Homogenize} and Proposition~\ref{prop:FrequentlyUsedIsomorphisms} (using Notation~\ref{not:ConingAbuse} for the second isomorphism).  The identity~\eqref{eq:EulerCharacteristic} applied to the Euler-Poincar\'e characteristic of $\calR/\calJ[\wDelta]$, coupled with Proposition~\ref{prop:FrequentlyUsedIsomorphisms}, gives
	\[
	\dim \spl_d^r(\Delta) =  \chi(\calR/\calJ[\wDelta],d) +\dim H_2(\calR/\calJ[\wDelta])_d-\dim H_1(\calR/\calJ[\wDelta])_d.
	\]
	To get Equation~\eqref{eq:dimformula}, note that $\calR$ has the form 
	$
	0\rightarrow \tS^{f_3}\rightarrow \tS^{f^\circ_2}\rightarrow \tS^{f^\circ_1}\rightarrow \tS^{f^\circ_0}\rightarrow 0;
	$
	taking the Euler characteristic in degree $d$ and using Equation~\eqref{eq:EulerCharESCC} yields Equation~\eqref{eq:dimformula}.  For Equation~\eqref{eq:celldimformula}, Proposition~\ref{prop:FrequentlyUsedIsomorphisms} implies that 
	$
	\dim \hspl^r_d(\Delta)=\dim \tS_d + \dim H_2(\calJ[\Delta])_d.
	$
	Taking the Euler-Poincar\'e characteristic of $\calJ[\Delta]$ gives
	\[
	\dim H_2(\calJ[\Delta])_d=\chi(\calJ[\Delta],d)+\dim H_1(\calJ[\Delta])_d- \dim H_0(\calJ[\Delta])_d.
	\]
	It is straightforward to show that $H_0(\calJ[\Delta])=0$; putting together the above two equations yields Equation~\eqref{eq:celldimformula}.
\end{proof}

\subsection{Generic simplicial complexes}\label{ss:genericvertexpositions}
It is well-known that, for a fixed $r$ and $d$, there is an open set in $(\R^{n})^{f_0}$ of vertex coordinates of $\Delta$ for which $\dim \spl^r_d(\Delta)$ is constant.  

\begin{definition}\label{def:GenericSimplicialComplex}
	Suppose $\Delta$ has vertex coordinates so that $\dim \spl^r_d(\Delta)\le \dim\spl^r_d(\Delta')$ for all simplicial complexes $\Delta'$ obtained from $\Delta$ by a small perturbation of the vertex coordinates.  Then we say $\Delta$ is \textit{generic} with respect to $r$ and $d$, or simply \textit{generic} if $r$ and $d$ are understood.
\end{definition}

Hence, for the purposes of obtaining a lower bound on $\dim\spl^r_d(\Delta)$, it suffices to obtain a lower bound on $\dim\spl^r_d(\Delta)$ when $\Delta$ is generic.

\section{Proof of Theorem~\ref{thm:LBGenericTet}: a lower bound in large degree}\label{sec:general_lowerbound}
To prove Theorem~\ref{thm:LBGenericTet} we use Equation~\eqref{eq:dimformula} from Proposition~\ref{prop:EulerCharacteristicAndDimension}, so we first describe how to compute the terms which appear in $\chi(\calJ[\wDelta],d)$.  From the discussion in Section~\ref{ss:genericvertexpositions}, it suffices to consider \textit{generic} \meshplural{}.  First, the Euler characteristic of $\calJ[\wDelta]$ has the form
\begin{equation*}\label{eq:EulerCharacteristicTrivariate}
	\chi(\calJ[\wDelta],d)=\sum_{\sigma\in\Delta^\circ_2} \dim \tJ(\wsigma)_d-\sum_{\tau\in\Delta^\circ_1} \dim\tJ(\wtau)_d+\sum_{\gamma\in\Delta^\circ_0}\dim \tJ(\wgamma)_d.
\end{equation*}
If $\Delta$ is a vertex star with $\gamma$ placed at the origin, we describe the effect which coning has on the vector spaces $\tJ(\beta)$, where $\beta$ is an $i$-face of $\Delta$.  The vector spaces $\tJ(\beta)\subset \tS$ and $\tJ(\wh{\beta})\subset \wS$ are related by tensor product. Explicitly,
$
\tJ(\wh{\beta})\cong \tJ(\beta)\otimes_\R \R[x_0]
$
and
\begin{equation}\label{eq:IdealHomogenize}
\dim \tJ(\wh{\beta})_d= \sum_{i\le d} \dim \tJ(\beta)_d.
\end{equation}
Hence to compute $\dim \tJ(\wh{\beta})_d$ it is necessary and sufficient to compute $\dim \tJ(\beta)_i$ for every $0\le i\le d$.  Since these dimensions are invariant under a translation of $\R^3$, we assume $\beta$ contains the origin and thus $\tJ(\beta)$ is graded.

\begin{proposition}\label{prop:edgeIdeals}
	Suppose $\Delta\subset\R^3$ is a \mesh{}, $r\ge 0$ is an integer, and $\tau\in\Delta_1$.  With $t_\tau,a_\tau,$ and $b_\tau$ as in Notation~\ref{not:EdgeData}, we have
	\[
	\dim \tJ(\tau)_d\le t_\tau\binom{d+1-r}{2}-a_\tau\binom{d+1-q_\tau}{2}-b_\tau\binom{d+2-q_\tau}{2},
	\]
	with equality if every triangle $\sigma$ containing $\tau$ has a distinct linear span (in particular, there is equality if $\Delta$ is generic).
\end{proposition}
\begin{proof}
	This is one of the fundamental computations for planar splines, originally due to Schumaker.  In its stated form, this formula is equivalent to a result of Schenck~\cite[Theorem~3.1]{RegSplines}.
\end{proof}

\begin{proposition}\cite[Corollary~3.18]{PaperA}\label{prop:vertexIdeals}
	Suppose $\Delta\subset\R^3$ is a generic closed vertex star with interior vertex $\gamma$, $r\ge 0$ is an integer, and $D_\gamma$ is the integer defined in~\eqref{eq:Dgamma}.  Then $\dim \tJ(\gamma)_d\le\binom{d+2}{2}$, with equality for $d>D_\gamma$.
\end{proposition}

\noindent For a \mesh{} $\Delta$ and vertex $\gamma\in\Delta_0$, we now relate $\LBcs(\Delta_\gamma,d,r)$ and $\LBos(\Delta_\gamma,d,r)$ from Equations~\eqref{eq:LBclosedstar} and~\eqref{eq:LBopenstar} (respectively) to the Euler characteristic of $\calJ[\Delta_\gamma]$.

\begin{proposition}\label{prop:EulerCharBounds}
	Let $\Delta$ be a generic \mesh{}.  If $\gamma\in\Delta^\circ_0$ then
	\begin{align*}
		\LBcs(\Delta_\gamma,d,r)= &\  2\ \binom{d+2}{2}+\chi\bigl(\calJ[\Delta_\gamma],d\bigr)-\dim\tJ(\gamma)_d\\[5 pt]
		= &\  \binom{d+2}{2}+\chi\bigl(\calJ[\Delta_\gamma],d\bigr)\mbox{\quad if }d>D_\gamma,
	\end{align*}
	where $\LBcs(\Delta_\gamma,d,r)$ is defined in Equation~\eqref{eq:LBclosedstar}.  If $\gamma$ is a boundary vertex of $\Delta$ then
	\[
	\LBos(\Delta_\gamma,d)=\binom{d+2}{2}+\chi(\calJ[\Delta_\gamma],d)\mbox{ for all } d\ge 0\ ,
	\]
	where $\LBos(\Delta_\gamma,d,r)$ is defined in Equation~\eqref{eq:LBopenstar}.
\end{proposition}
\begin{proof}
	If $\gamma$ is an interior vertex then $\Delta_\gamma$ is a closed vertex star $\calJ[\Delta_\gamma]$ has the form
	$
	0\rightarrow \bigoplus_{\sigma\in\Delta^\circ_2} \tJ(\sigma)\rightarrow \bigoplus_{\tau\in\Delta^\circ_1} \tJ(\tau) \rightarrow \tJ(\gamma)\rightarrow 0.
	$
	Taking the graded Euler characteristic, the first equation now follows from the fact that $\dim\tJ(\sigma)_d=\binom{d+1-r}{2}$, Proposition~\ref{prop:edgeIdeals}, and Proposition~\ref{prop:vertexIdeals}.  If $\gamma$ is a boundary vertex then $\Delta_\gamma$ is an open vertex star and $\calJ[\Delta_\gamma]$ has the form
	$
	0\rightarrow \bigoplus_{\sigma\in\Delta^\circ_2} \tJ(\sigma)\rightarrow \bigoplus_{\tau\in\Delta^\circ_1} \tJ(\tau) \rightarrow 0.
	$
	Taking the graded Euler characteristic, the first equation now follows from the fact that $\dim\tJ(\sigma)_d=\binom{d+1-r}{2}$ and Proposition~\ref{prop:edgeIdeals}.
\end{proof}

\begin{proposition}\label{prop:DimensionsCountsForTetrahedralComplexes}
	Suppose $\Delta\subset\R^3$ is a \mesh{}, $r\ge 0$ is an integer, $\sigma\in\Delta_2$, $\tau\in\Delta_1$, and $\gamma\in\Delta^\circ_0$.  Then
	\begin{equation}\label{eq:4varcodim1}	
	\dim \wS_d=  \binom{d+3}{3}\ ,\quad 
	\dim \tJ(\wsigma)_d=  \binom{d+2-r}{3}\ ,
	\end{equation}
	
	\begin{align}
		\dim \tJ(\wtau)_d\le \ & t_\tau\binom{d+2-r}{3}-a_\tau\binom{d+2-q_\tau}{3}-b_\tau\binom{d+3-q_\tau}{3}\ ,	\label{eq:4varcodim2}\\[3 pt]
		\dim \tJ(\wgamma)_d=\  & \binom{d+3}{3}+\sum\limits_{i=0}^{D_\gamma} \left(\dim\tJ(\gamma)_i-\binom{d+2}{2}\right)\ , \label{eq:4varcodim3}
	\end{align}
	$\dim H_2(\calR/\calJ[\wDelta])_d=C$ for some positive integer $C$ and $d\gg0$, and $\dim H_1(\calR/\calJ[\wDelta])_d=  0$ for  $d\gg 0$.	If $\Delta$ is generic then~\eqref{eq:4varcodim2} is an equality.
\end{proposition}
\begin{proof}
	Equations~\eqref{eq:4varcodim1} are straightforward to derive.  Equations~\eqref{eq:4varcodim2} and~\eqref{eq:4varcodim3} follow from Propositions~\ref{prop:edgeIdeals} and~\ref{prop:vertexIdeals}, respectively, using Equation~\eqref{eq:IdealHomogenize}.  
	It follows from~\cite[Lemma~3.1]{Spect} that $\dim H_2(\calR/\calJ[\wDelta])_d=C$ for a positive integer $C$ and $H_1(\calR/\calJ[\wDelta])$ vanishes in large degree.
\end{proof}
We now provide a lower bound on the integer $C$ satisfying $\dim H_2(\calR/\calJ[\wDelta])_d=C$ for $d\gg 0$ (see Proposition~\ref{prop:DimensionsCountsForTetrahedralComplexes}).  The key is to describe the effect of the coning construction $\Delta\to\wDelta$ on the homology module $H_2(\calR/\calJ[\Delta])$ in large degree.

\begin{proposition}\label{prop:AssHom}
	Let $\Delta\subset\R^3$ be a \mesh{}.  Then, for $d\gg 0$,
	\begin{align*}
		\dim H_2(\calR/\calJ[\wDelta])_d= & \sum_{\gamma\in\Delta_0}\sum_{i \ge 0} \left( \dim\hspl^r_i(\Delta_\gamma)-\chi\bigl(\calR/\calJ[\Delta_\gamma],i\bigr) \right)\\
		=& \sum\limits_{\gamma\in\Delta_0}\sum\limits_{i=0}^{3r+1} \left( \dim\hspl^r_i(\Delta_\gamma)-\dbinom{i+2}{2}-\chi\bigl(\calJ[\Delta_\gamma],i\bigr) \right)\\
		\ge & \sum\limits_{\gamma\in\Delta_0}\sum\limits_{i=0}^{3r+1}[-\chi(\calJ[\Delta_\gamma],i)]_+
	\end{align*}
\end{proposition}
\begin{proof}
	The first equality is~\cite[Corollary~9.2]{AssHom}.  For the second equality,
	\[
	\dim H_1(\calJ[\Delta_\gamma])_i=\dim\hspl^r_i(\Delta)-\binom{d+2}{2}-\chi(\calJ[\Delta_\gamma],i)
	\]
	is an immediate consequence of Equation~\eqref{eq:celldimformula}.  It follows from the main result of~\cite{ANS96} (see also~\cite{RegSplines}) that $\dim\hspl^r_i(\Delta_\gamma)=\binom{i+2}{2}+\chi(\calJ[\Delta_\gamma],i)$ for $i\ge 3r+2$\ .  In other words, $H_1(\calJ[\Delta_\gamma])_i=0$ for $i\ge 3r+2$.  The final inequality follows from the fact that $\hspl^r_i(\Delta_\gamma)$ always contains the space of global homogeneous polynomials of degree $i$, which has dimension $\binom{i+2}{2}$\ .
\end{proof}

We prove in~\cite{PaperA} the following modification of a result of Whiteley~\cite{WhiteleyComb}.

\begin{theorem}\cite[Theorem~1.3]{PaperA}\label{thm:WhitelyGenericLowDegree}
	If $\Delta$ is a generic closed star with interior vertex $\gamma$, then $\dim \hspl^r_d(\Delta)=\binom{d+2}{2}$ for $d\le D_\gamma$.
\end{theorem}

\begin{corollary}\label{cor:H1ExactLowDegree}
	If $\Delta$ is a generic closed star with interior vertex $\gamma$, then $\dim H_1(\calJ[\Delta])_d=-\chi(\calJ[\Delta],d)$ \ for\  $d\le D_\gamma$.
\end{corollary}
\begin{proof}
	Immediate from Equation~\eqref{eq:celldimformula} and Theorem~\ref{thm:WhitelyGenericLowDegree}.
\end{proof}

\begin{proof}[Proof of Theorem~\ref{thm:LBGenericTet}]
	Since $H_0(\calJ[\wDelta])_d=0$ for $d\gg 0$ by Proposition~\ref{prop:DimensionsCountsForTetrahedralComplexes}, then~\eqref{eq:dimformula} implies that for $d\gg 0$,  
	\begin{equation}\label{eq:dimformulalargedegree}
	\dim \spl^r_d(\Delta)=(f_3-f^\circ_2+f^\circ_1-f^\circ_0)\binom{d+3}{3}+\chi(\calJ[\wDelta],d)+H_1(\calR/\calJ[\wDelta])_d\ .
	\end{equation}
	Hence it suffices to prove that \ 
	$ \displaystyle
	\LB(\Delta,d,r)\le \bigl(f_3-f^\circ_2+f^\circ_1-f^\circ_0\bigr)\binom{d+3}{3}+\chi\bigl(\calJ[\wDelta],d\bigr)+C\,,
	$
	where $C=H_1(\calR/\calJ[\wDelta])_d$ for $d\gg 0$.  
	Put 
	\[
	\chi'(d) := \left(f^\circ_2-\sum\limits_{\tau\in\Delta^\circ_1} t_\tau \right)\binom{d+2-r}{3} +\sum\limits_{\tau\in\Delta^\circ_1} \left( a_\tau\binom{d+2-q_\tau}{3}+b_\tau\binom{d+3-q_\tau}{3}\right)\ ,
	\]
	so $\chi(\calJ[\wDelta],d)=\chi'(d)+\sum_{\gamma\in\Delta^\circ_0} \dim\tJ(\wgamma)_d$ by Equation~\eqref{eq:EulerCharacteristicTrivariate} and Proposition~\ref{prop:DimensionsCountsForTetrahedralComplexes}.
	Another application of Proposition~\ref{prop:DimensionsCountsForTetrahedralComplexes} gives
	\begin{equation}\label{eq:tetchiineq}
	\chi(\calJ[\wDelta],d)=\chi'(d)+f^\circ_0\binom{d+3}{3}+\sum_{\gamma\in\Delta^\circ_0}\sum_{i=0}^{D_\gamma}\left(\dim \tJ(\gamma)_i-\binom{i+2}{2}\right)\ .
	\end{equation}
	Now, by Proposition~\ref{prop:AssHom}, $\dim H_1(\calJ[\wDelta])_d\ge \sum\limits_{\gamma\in\Delta_0}\sum\limits_{i=0}^{3r+1}[-\chi(\calJ[\Delta_\gamma],i)]_+$
	for $d\gg 0$.  Corollary~\ref{cor:H1ExactLowDegree} allows us to remove the $_+$ from the summation for interior vertices in the range $0\le i\le D_\gamma$:
	\begin{multline}\label{eq:tetH1ineq}
		\dim H_1(\calJ[\wDelta])_d \ge \sum\limits_{\gamma\in\Delta^\circ_0}\left( \sum\limits_{i=0}^{D_\gamma} [-\chi(\calJ[\Delta_\gamma],i)] +\sum\limits_{i=D_\gamma+1}^{3r+1} [-\chi(\calJ[\Delta_\gamma],i)]_+ \right)\\ +\sum\limits_{\gamma\in\Delta^\circ\setminus\Delta^\circ_0}\sum_{i=0}^{3r+1} [-\chi(\calJ[\Delta_\gamma],i)]_+ \ ,
	\end{multline}
	for $d\gg 0$.  Combining~\eqref{eq:tetchiineq} and~\eqref{eq:tetH1ineq} with~\eqref{eq:dimformulalargedegree} yields
{\small	\begin{align*}
		\dim\spl^r_d(\Delta)\geq \ & \bigl(f_3-f^\circ_2+f^\circ_1\bigr)\binom{d+3}{3} +\chi'(d)\\
		& +\sum\limits_{\gamma\in\Delta^\circ_0} \left(\sum\limits_{i=0}^{D_\gamma} \left[\dim \tJ(\gamma)_i-\binom{i+2}{2}-\chi(\calJ[\Delta_\gamma],i)\right] +\sum\limits_{i=D_\gamma+1}^{3r+1} \bigl[-\chi(\calJ[\Delta_\gamma],i)\bigr]_+ \right)\\
		&+\sum\limits_{\gamma\in\Delta^\circ\setminus\Delta^\circ_0}\sum\limits_{i=0}^{3r+1} \bigl[-\chi(\calJ[\Delta_\gamma],i)\bigr]_+
	\end{align*}}
	for $d\gg 0$.  By Proposition~\ref{prop:EulerCharBounds}, if $\gamma\in\Delta^\circ_0$,
	\begin{align*}
		\dbinom{i+2}{2}-\LBcs(\Delta_\gamma,i,r)= & \tJ(\gamma)_i-\binom{i+2}{2}-\chi(\calJ[\Delta_\gamma],i)\ , \text{ and }\\
		\dbinom{i+2}{2}-\LBcs(\Delta_\gamma,i,r)= & -\chi(\calJ[\Delta_\gamma],i) \mbox{ for } i>D_\gamma\ .
	\end{align*}
	Also by Proposition~\ref{prop:EulerCharBounds}, if \  $\gamma\in\Delta_0\setminus\Delta^\circ_0$ \ then \ 
	$
	\binom{i+2}{2}-\LBos(\Delta_\gamma,i,r)=-\chi\bigl(\calJ[\Delta_\gamma],i\bigr)\ .$ Thus, 
{\small	\begin{align*}
		\dim\spl^r_d(\Delta)\ge \ & \bigl(f_3-f^\circ_2+f^\circ_1\bigr)\binom{d+3}{3} +\chi'(d)\\
		& +\sum\limits_{\gamma\in\Delta^\circ_0} \left(\sum\limits_{i=0}^{D_\gamma} \left[\dbinom{i+2}{2}-\LBcs(\Delta_\gamma,i,r)\right] +\sum\limits_{i=D_\gamma+1}^{3r+1} \left[\dbinom{i+2}{2}-\LBcs(\Delta_\gamma,i,r)\right]_+ \right)\\
		&+\sum\limits_{\gamma\in\Delta^\circ\setminus\Delta^\circ_0}\sum\limits_{i=0}^{3r+1} \left[\dbinom{i+2}{2}-\LBos(\Delta_\gamma,i,r)\right]_+\\
		=\ & (f_3-f^\circ_2+f^\circ_1)\binom{d+3}{3} +\chi'(d)-f^\circ_0\binom{r+3}{3}\\
		& +\sum\limits_{\gamma\in\Delta^\circ_0} \left(\sum\limits_{i=r+1}^{D_\gamma} \left[\dbinom{i+2}{2}-\LBcs(\Delta_\gamma,i,r)\right]+\sum\limits_{i=D_\gamma+1}^{3r+1} \left[\dbinom{i+2}{2}-\LBcs(\Delta_\gamma,i,r)\right]_+ \right)\\
		&+\sum\limits_{\gamma\in\Delta^\circ\setminus\Delta^\circ_0}\sum\limits_{i=r+1}^{3r+1} \left[\dbinom{i+2}{2}-\LBos(\Delta_\gamma,i,r)\right]_+\\
		=\ & (f_3-f^\circ_2+f^\circ_1)\binom{d+3}{3} +\chi'(d)-f^\circ_0\binom{r+3}{3}+\sum\limits_{\gamma\in\Delta_0} N_\gamma \ =\   \LB(\Delta,d,r)\ ,
	\end{align*}}
	where $N_\gamma$ is defined in~\eqref{eq:Ngamma} and $\LB(\Delta,d,r)$ is defined in~\eqref{eq:LBtetrahedral}.
\end{proof}

\section{Examples}\label{sec:Examples}
In this section we compare our lower bounds with those by Alfeld and Schumaker in \cite{Tri} and Mourrain and Villamizar \cite{D3}.  Except for the non-simplicial partition in Example~\ref{subsec:cube}, the other examples appear in~\cite{Tri}.  It is well-known that for $d\gg 0$, $\dim \spl^r_d(\Delta)$ is a \textit{polynomial} function.  That is, there is a polynomial in $d$ with rational coefficients, which we denote by $P^r_d(\Delta)$, so that $\dim \spl^r_d(\Delta)=P^r_d(\Delta)$ for $d\gg 0$.  (In commutative algebra this is called the \textit{Hilbert polynomial} of $\spl^r(\wDelta)$ -- see Remark~\ref{rem:HilbertPolynomial}.)  We can compute both the exact dimension $\dim \spl^r_d(\Delta)$ and the polynomial $P^r_d(\Delta)$ in Macaulay2~\cite{M2} using the \texttt{Algebraic Splines} package.  We give the computations of $P^r_d(\Delta)$ in Sections~\ref{ss:m-s},\ref{ss:ms-cavity},\ref{ss:torus}, and~\ref{subsec:cube} for generic vertex positions of the examples.  The exact generic dimension $\dim \spl^r_d(\Delta)$ for our examples is shown in the column labeled `gendim' in Tables~\ref{tbl:m-s},\ref{tbl:ms-and-torus}, and \ref{tbl:cube}.  The lower bound from Theorem~\ref{thm:LBGenericTet} is in the column labeled $\LB(d)$, and lower bounds from the literature appear in columns labeled $\LB$ with an appropriate citation.

\subsection{Three dimensional Morgan-Scott}\label{ss:m-s}
Let $\Delta$ be the simplicial complex in Fig.~\ref{fig:3DMS} from Section~\ref{exam:intro}. 
In Table \ref{tbl:m-s} we record the values of the lower bounds on $\dim\spl^r_d(\Delta)$ for $r=3$ and $r=4$.  In column 3 we give the dimension of the space of polynomials of degree at most $d$ (this is $\binom{d+3}{3}$), in columns 4--6 the bounds are obtained by applying the formulas proved in \cite[Theorem 5.1]{D3}, \cite[Example 8.2]{Tri}, and Theorem \ref{thm:LBGenericTet}, respectively.  The last column records the value for the exact dimension for the given order of continuity $r$ and degree $d$. For $d\gg 0$, the lower bounds can be computed as in Example~\ref{exam:intro} and are given by 
\[
\LB(\Delta,d,3) = \frac{5}{3}d^3-41d^2+\frac{451}{2}d-323 \quad \text{and}
\quad
\LB(\Delta,d,4) = \frac{5}{2}d^3-55d^2+\frac{807}{2}d-803\ .
\]
These coincide with the polynomials $P^3_d(\Delta)$ and $P^4_d(\Delta)$, respectively.
\begin{table}
	\centering
	\renewcommand{\arraystretch}{1.2}
	\begin{tabular}{c|c||c|c|c|c|c}
		$r$&$d$&$\binom{d+3}{3}$& \LB \cite{D3} & \LB \cite{Tri}  &$\LB(d)$&\mbox{gendim}\\
		
		\hline
		
		3 & 8 &  165   &  165   &  165    &  137   &  165  \\
		3 & 9 &  220   &  220   &  220    &  208   &  220  \\
		3 & \pmb{10} & 286   &  286   &  286    &  332   &  332  \\
		3 & 11 & 364   &  364   &  364    &  524   &  524  \\
		3 & 12 & 455   &  591   &  593    &  799   &  799  \\
		3 & 13 & 560   &  964   &  948    &  1172   &  1172  \\
		
		\hline
		
		4 & 11  &   364    &  364   &  364     &  308    & 364  \\
		4 & 12  &   455    &  455   &  455     &  439    & 455  \\
		4 & \pmb{13}  &   560    &  560   &  560     &  640    & 640  \\
		4 & 14  &   680    &  680   &  680     &  926    & 926  \\
		4 & 15  &   680    &  896   &  832     & 1312    & 1312 \\
	\end{tabular}
	\caption{Lower bounds on $\dim \spl^r_d(\Delta)$, where $\Delta$ is the three dimensional Morgan-Scott partition in Fig.~\ref{fig:3DMS}; see Example~\ref{ss:m-s}.  The initial degree is bolded.}\label{tbl:m-s}
\end{table}

\subsection{Morgan-Scott with a cavity}\label{ss:ms-cavity}
We consider $\Delta$ as the partition obtained by removing the central tetrahedron in Fig.~\ref{fig:3DMS}. 
In Table~\ref{tbl:ms-cavity} we list the values of the lower bound in Theorem~\ref{thm:LBGenericTet} applied for $r=1,\dots,4$ along with those presented in~\cite[Example~8.4]{Tri}.
For this partition we have $f_3= 14$ tetrahedra, $f_2^\circ=24$, $f_2^\circ=12$, and $f_0^\circ=0$. Applying \eqref{eq:LBtetrahedral} in Theorem~\ref{thm:LBGenericTet} we get 
\[
\LB(\Delta,d,1) = \frac{7}{3}d^3 - 10 d^2+\frac{41}{3}d+2\ ,
\quad 
\LB(\Delta,d,2) = \frac{7}{3}d^3 - 22d^2+\frac{185}{3}d-10\ , 
\]
\[
\LB(\Delta,d,3) = \frac{7}{3}d^3 - 34d^2+\frac{473}{3}d - 142\ ,\text{ and} 
\quad
\LB(\Delta,d,4) = \frac{7}{3}d^3 - 46d^2+\frac{869}{3}d- 406\ .
\]
As shown in Table~\ref{tbl:ms-cavity}, for $r=1,\dots, 4$, the bound $\LB(\Delta,d,r)$ gives the exact dimension of $\spl^r_d(\Delta)$ beginning at the  the initial degree of $\spl^r(\Delta)$.  Hence the polynomials $\LB(\Delta,d,r)$ coincide with the polynomials $P^r_d(\Delta)$ for $r=1,2,3,4$.
\subsection{Square--shaped torus}\label{ss:torus}
We consider the tetrahedral decomposition of the square-shaped torus depicted on the left in Fig.~\ref{fig:cube}.  This is composed of four three-dimensional `trapezoids,' each of which is split into six tetrahedra along an interior diagonal.  We have $f_3=24$, $f^\circ_2=32$, $f^\circ_1=8$, and $f^\circ_0=0$.  An explicit set of faces and coordinates is provided in \cite[Example 8.3]{Tri}.  In Table~\ref{tbl:ms-and-torus} we list the values of the lower bound of Theorem~\ref{thm:LBGenericTet} applied for $r=1,\ldots,4$ along with those presented in \cite[Example~8.3]{Tri}.
We have, 
\[
\LB(\Delta,d,1) = 4d^3 - 8 d^2 + 4d\ , 
\quad
\LB(\Delta,d,2) = 4d^3 - 24d^2 + 44d - 24\ , 
\]
\[\LB(\Delta,d,3) = 4d^3 - 40d^2 + 128d - 132\ ,\text{ and} 
\quad
\LB(\Delta,d,4) = 4d^3 - 56d^2 + 25d - 360\ .
\]
Again, the polynomials $\LB(\Delta,d,r)$ coincide with $P^r_d(\Delta)$ for $r=1,2,3, 4$.

\begin{table}[htbp]
	\centering
	\begin{tabular}{ccc}
		\begin{subtable}{0.2\linewidth}	
			\centering	
			\renewcommand{\arraystretch}{1.2}	
			\begin{tabular}{c|c||c}
				$r$&$d$&$\binom{d+3}{3}$ \\
				
				\hline
				
				1 & 2 &   10   \\
				1 & 3 &   20   \\
				1 & 4 &   35   \\
				1 & 5 &   56   \\
				
				\hline 

				2 & 4 &   35    \\
				2 & 5 &   56    \\
				2 & 6 &   84    \\
				2 & 7 &   120   \\
				2 & 8 &   165   \\
				2 & 9 &   220   \\
				
				\hline 
				
				3 & 6 &   84      \\
				3 & 7 &   120     \\
				3 & 8 &   165     \\
				3 & 9 &   220     \\
				3 & 10 &  286     \\
				3 & 11 &  364     \\
				
				\hline 
				
				4 & 8 &   165  \\
				4 & 9 &   220  \\
				4 & 10 &  286  \\
				4 & 11 &  364  \\
				4 & 12 &  455  \\
				4 & 13 &  560  \\
			\end{tabular}
			\caption*{ }
		\end{subtable}
		&
		\begin{subtable}{0.32\linewidth}
			\centering	
			\renewcommand{\arraystretch}{1.2}	
			\begin{tabular}{c|c|c}
				\LB \cite{Tri} &$\LB(d)$&{\mbox{gendim}}\\
				
				\hline
				
				10  &    8   &    10      \\
				20  &   16   &    20      \\
				46  &	 46   &    46      \\
				112  &  112   &   112      \\
				
				\hline 
				
				35  &    34    &    35    \\
				56  &	  40    &    56    \\
				84  &    72    &    84    \\
				120  &   144    &   144    \\
				242  &   270    &   270    \\
				436  &   464    &   464    \\
				
				\hline 
				
				84  &    84   &   84     \\
				120  &    96   &  120     \\
				165  &   138   &  165     \\
				220  &   224   &  224     \\
				286  &   368   &  368     \\
				428  &   584   &  584     \\
				
				\hline 
				
				165  &    162    &   165   \\
				220  &    176    &   220   \\
				286  &    224    &   286   \\
				364  &    320    &   364   \\
				455  &    478    &   478   \\
				560	 &    712    &   712   \\
				
			\end{tabular}
			\caption{Morgan-Scott with cavity.}\label{tbl:ms-cavity}
		\end{subtable}
		&
		\begin{subtable}{0.35\linewidth}
			\centering	
			\renewcommand{\arraystretch}{1.2}	
			\begin{tabular}{c|c|c}
				LB\cite{Tri} &$\LB(d)$&{\mbox{gendim}}\\
				
				\hline
				
				10  &     8    &    10      \\
				48  &    48    &    48      \\
				144  &	 144    &   144      \\
				320  &   320    &   320      \\
				
				\hline 
				
				35  &   24    &    35     \\
				93  &	  96    &    96     \\
				237  &  240    &   240     \\
				477  &  480    &   480     \\
				837  &  840    &   840     \\
				1341  & 1344    &  1344    \\

				\hline 
				
				84  &    60   &    84     \\
				151  &   176   &   176     \\
				351  &   380   &   380     \\
				663  &   696   &   696     \\
				1111  &  1148   &  1148     \\
				1719  &  1760   &  1760     \\

				\hline 
				
				165  &   120   &   165   \\
				220  &   288   &   288   \\
				483  &   560   &   560   \\
				875  &   960   &   960   \\
				1419  &  1512   &  1512   \\
				2139  &  2240   &  2240   \\
			\end{tabular}
			\caption{Square--shaped torus.}\label{tbl:torus}
		\end{subtable}
	\end{tabular}
	\caption{Lower bounds for the partitions in Examples~\ref{ss:ms-cavity} 
		(Table \ref{tbl:ms-cavity}) and \ref{ss:torus} (Table \ref{tbl:torus}).} \label{tbl:ms-and-torus}.
\end{table}

\subsection{Non-simplicial partition}\label{subsec:cube}
For the sake of simplicity we have limited our discussion to \meshplural{}, but our lower bound works for polytopal partitions with one important modification.  That is, the sum in the definition of $N_\gamma$ should not stop in degree $3r+1$, but should continue until all positive contributions are accounted for (in~\cite{RegSplines} a bound is given that could be taken for the upper limit of this sum, but in practice one should simply stop as soon as the contributions switch from positive to negative).  

We compute the bound of Theorem~\ref{thm:LBGenericTet} for the polytopal partition $\Delta$ in Fig.~\ref{fig:cube}, which is a polytopal analog of the three-dimensional Morgan-Scott partition.  
It consists of a cube inside of which we place its dual polytope (the octahedron).  
Then the partition consists of the interior octahedron along with the convex hull of pairs of dual faces.  For example, each vertex of the inner octahedron is paired with a dual square face of the cube and their convex hull is a square pyramid.  The number of interior vertices is $f_0^\circ = 6$,  the number of interior edges is $f_1^\circ = 36$, and the number of interior two-faces if $f_2^\circ = 56$.
Each  interior vertex $\gamma$ is connected by an edge to eight vertices i.e., 
$f_1^\circ (\Delta_\gamma) = 8$ and $f_2^\circ (\Delta_\gamma)=16$  in the star $\Delta_\gamma$.
Thus, $D_\gamma = \bigl\lfloor \frac{3r+1}{2}\bigr\rfloor$ for all $\gamma\in \Delta_0^\circ$.
There are eight vertices $\gamma'$ on the boundary, for each of them we have  $f_2^\circ (\Delta_{\gamma'})=9$, and $f_1^\circ (\Delta_{\gamma'})=3$ in the open stars $\Delta_{\gamma'}$.  
Applying \eqref{eq:LBclosedstar} and \eqref{eq:LBopenstar} yields
\begin{align*}
	\LBcs(\Delta_\gamma, d,r)&= 2\binom{d+2}{2}+\bigl(16-8 t_\tau \bigr)\binom{d+1-r}{2} +  8a_\tau\binom{d+1-q_\tau}{2}+8b_\tau\binom{d+2-q_\tau}{2},
	\intertext{ and}
	\LBos(\Delta_{\gamma'},d,r) &= \binom{d+2}{2}+\bigl(9-3 t_\tau\bigr)\binom{d+1-r}{2}
	+3a_\tau\binom{d+1-q_\tau}{2}+3b_\tau\binom{d+2-q_\tau}{2}.
\end{align*}

\noindent By Theorem \ref{thm:LBGenericTet} the dimension of the spline space then $\dim \spl^r_d(\Delta)\ge \LB(d)$ for   $d\gg 0$, where 

\begin{equation}\label{eq:LB_cube}
\LB(\Delta,d,r) = 7\binom{d+3}{3}+ (56-36\cdot 2)\binom{d+2-r}{3}
+36\binom{d+1}{3}-6\binom{r+3}{3}+6N_{\gamma}+8N_{\gamma'}\ .
\end{equation}

\noindent Every edge $\tau\in\Delta_1^\circ$ is in four two-dimensional faces i.e., $n_\tau=4$. This leads to three values of $t_\tau$: if $r=0$ then $t_\tau =2$; if $r=1$ then $t_\tau=2$; if $r\geq 2$ then $t_\tau=4$\ .

\begin{figure}[htbp]
	\centering
	\includegraphics[scale=0.5]{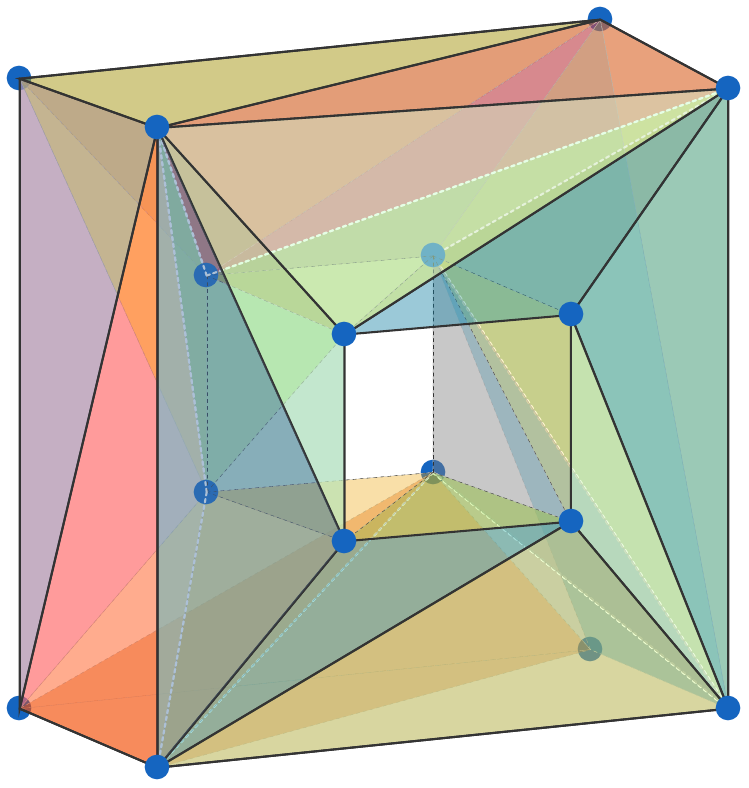}\quad
	\includegraphics[scale=0.78]{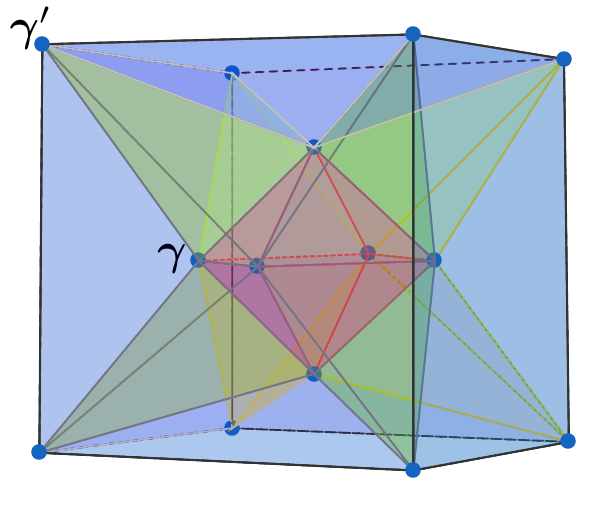}
	\caption{Square--shaped torus in Example \ref{ss:torus} (left), and the non-simplicial polyhedral partition in Example \ref{subsec:cube} (right).}\label{fig:cube}
\end{figure}

\noindent \emph{Case 1.} Let $r=0$, then $t_\tau=2$, $q_\tau =2$,  $a_\tau=0$, and $b_\tau=1$  for all $\tau\in\Delta_1^\circ$, and $D_\gamma = 0$ for all $\gamma\in \Delta_0^\circ$. It follows,
\begin{align*}
	\LBcs(\Delta_\gamma,d,0)&= 2\binom{d+2}{2} +  8\binom{d}{2}, \text{ and \;}\\
	\LBos(\Delta_{\gamma'},d,0)&= \binom{d+2}{2}+(9-3\cdot 2)\binom{d+1}{2}
	+3\binom{d}{2}.
\end{align*}
From \eqref{eq:Ngamma}, we have $N_\gamma= N_{\gamma'}= 0$. Therefore,
\begin{equation}\label{eq:cube0}
\LB(\Delta,d,0) = 7\binom{d+3}{3}+ (56-36\cdot 2)\binom{d+2}{3}
+36\binom{d+1}{3}-6\ 
\ = \ \frac{9}{2} d^3 - d^2 +\frac{3}{2}d +1. 
\end{equation}
\noindent \emph{Case 2.} If $r=1$, then $t_\tau= 3$, $q_\tau= 3$, $a_\tau=0$, and $b_\tau=2$ for all $\tau\in\Delta_1^\circ$, and $D_\gamma= 2$ for all $\gamma\in\Delta_0^\circ$. It follows, 
\begin{align*}
	\LBcs(\Delta_\gamma,d,1)&= 2\binom{d+2}{2} + (16-8\cdot 3)\binom{d}{2}+  8\cdot 2 \binom{d-1}{2}\ ,	 \text{  and \quad}\\
	\LBos(\Delta_{\gamma'},d,1)&= \binom{d+2}{2}+3\cdot 2\binom{d-1}{2}\ .
\end{align*}
From \eqref{eq:Ngamma} we have $N_\gamma= 2$ and $N_{\gamma'}= 0$. Therefore,
\begin{align}
	\LB(\Delta,d,1) & = 7\binom{d+3}{3}+ (56-36\cdot 3)\binom{d+1}{3}
	+36\cdot 2\binom{d}{3}-6\cdot 4 +6\cdot 2\nonumber\\
	& = \ \frac{9}{2} d^3 - 29d^2 +\frac{91}{2}d -5\ . \label{eq:cube1}
\end{align}
\noindent \emph{Case 3.} For every $r\geq 2$, we have $t_\tau=4$. 
We write the explicit formula for $r=2$, the other cases follow similarly. 
We have $q_\tau= 4,\  a_\tau=0$ and $b_\tau=3$ for all $\tau\in \Delta_1^\circ$, and $D_\gamma = 3$. Then,
\begin{align*}
	\LBcs(\Delta_\gamma,d,2)&= 2\binom{d+2}{2} + (16-8\cdot 4)\binom{d-1}{2}+  8 \cdot 3 \binom{d-2}{2}\ ,	 \text{ \; and \quad}\\
	\LBos(\Delta_{\gamma'},d,2)&= \binom{d+2}{2}+(9-3\cdot 4)\binom{d-1}{2}+3\cdot 3\binom{d-2}{2}\ .
\end{align*}
From \eqref{eq:Ngamma} we have $N_\gamma= 18$ and $N_{\gamma'}= 3$. Therefore,
\begin{align}
	\LB(\Delta,d,2) &= 7\binom{d+3}{3}+ (56-36\cdot 4)\binom{d}{3}
	+36\cdot 3\binom{d-1}{3}-6\cdot 10 +6\cdot 18+8\cdot 3\nonumber\\
	& = \frac{9}{2} d^3 - 57d^2 +\frac{363}{2}d -29\ .\label{eq:cube2}
\end{align}
The bounds \eqref{eq:cube0}, \eqref{eq:cube1}, and \eqref{eq:cube2} are the polynomials $P^1_d(\Delta),P^2_d(\Delta),$ and $P^3_d(\Delta)$, respectively. In Table \ref{tbl:cube} we record the values fo $\LB(\Delta,d,r)$ along with the lower bound obtained in \cite{D3}. 
\begin{table}
	\centering
	\renewcommand{\arraystretch}{1.2}
	\begin{tabular}{c|c||c|c|c|c}
		$r$&$d$&$\binom{d+3}{3}$& \LB \cite{D3}  &$\LB(d)$&\text{gendim}\\
		
		\hline
		
		1 & 2 &   10    &  10   &    6    &    10      \\
		1 & 3 &   20    &  20	&   -8    &    20      \\
		1 & 4 &   35    &  35   &	 1    &    35      \\
		1 & $\pmb{5}$ &   56    &  56   &   60    &    60      \\
		1 & 6 &   84    & 160   &  196    &   196      \\
		
		\hline
		
		2 & 8 &  165    &   165   &     79    &   165     \\
		2 & $\pmb{9}$ &  220    &   220   &    268    &   268     \\
		2 & 10 & 286    &   352   &    586    &   586     \\
		2 & 11 & 364    &   826   &   1060    &  1060     \\
		2 & 12 & 455    &  1483   &   1717    &  1717     \\
		2 & 13 & 560    &  2350   &   2584    &  2584     \\
		2 & 14 & 680    &  3454   &   3688    &  3688     \\		
		
		\hline
		
		3 & 11 & 364    &  364   &     148   &   364     \\
		3 & 12 & 455    &  455   &     425   &   455     \\
		3 &$\pmb{13}$ & 560    &  560   &     856   &   856     \\
		3 & 14 & 680    &  988   &    1468   &  1468     \\	
		3 & 15 & 816    &  1808  &    2288   &  2288     \\
		3 & 16 & 969    &  2863  &    3343   &  3343     \\
	\end{tabular}
	\caption{Bounds for the non-simplicial partition in Example~\ref{subsec:cube},  Fig.~\ref{fig:cube} (right).}\label{tbl:cube}
\end{table}

\section{Concluding Remarks}\label{sec:ConcludingRemarks}

\begin{remark}\label{rem:HilbertPolynomial}
	The dimension $\dim \spl^r_d(\Delta)$ of splines on $\Delta$ is a polynomial in $d$ when $d\gg 0$; this polynomial is known as the \textit{Hilbert polynomial} of $\spl^r(\wDelta)$ in algebraic geometry.  Theorem~\ref{thm:LBGenericTet} gives a lower bound on the Hilbert polynomial of $\spl^r(\wDelta)$.  For some value of $d$, $\dim \spl^r_d(\Delta)$ will begin to agree with the Hilbert polynomial.  
	In algebraic geometry there is an integer which bounds when $\dim \spl^r_d(\Delta)$ becomes polynomial, known as the \textit{Castelnuovo-Mumford regularity} of $\spl^r(\wDelta)$.  It would be interesting to bound the the regularity of $\spl^r(\wDelta)$ for \meshplural{}, perhaps by extending methods from~\cite{RegSplines}.
\end{remark}

\begin{remark}\label{rem:AS8r+1}
	We suspect that our formula in Theorem~\ref{thm:LBGenericTet} is a lower bound on $\dim \spl^r_d(\Delta)$ for $d\ge 8r+1$ by the following reasoning.
	In~\cite[Theorem~24]{LocSup}, Alfeld, Schumaker, and Sirvent prove that
	$
	\dim \spl^r_d(\Delta)=\sum_{\beta\in\Delta}|\mathcal{D}(\beta)|
	$
	for $d\ge 8r+1$, where the sum runs across all simplices $\beta\in\Delta$ and $\mathcal{D}(\beta)$ is a \textit{minimal determining set} for the simplex $\beta$.    
	Counting the size of the sets $|\mathcal{D}(\beta)|$ gives rise to expressions using binomial coefficients using the same Convention~\ref{conv:BinomialCoefficients}.  For $r=1$ these are counted explicitly in~\cite{ASWTet}, while counts for more general $r$ (with supersmoothness) may be found in~\cite{AS89}.  We expect that for a fixed $r$ and $d\ge 8r+1$, $|\mathcal{D}(\beta)|$ is a polynomial of degree $\dim \beta$ for all $\beta\in\Delta$.  If so, then $\sum_{\beta\in\Delta}|\mathcal{D}(\beta)|$ is a polynomial for $d\ge 8r+1$, and this is the Hilbert polynomial of $\spl^r(\wDelta)$.  Since the formula in Theorem~\ref{thm:LBGenericTet} is a lower bound on the Hilbert polynomial of $\spl^r(\wDelta)$ (see Remark~\ref{rem:HilbertPolynomial}), it would follow that it is a lower bound on $\dim \spl^r_d(\Delta)$ for $d\ge 8r+1$.  It would also be interesting to know if~\cite{LocSup} has implications for the regularity of $\spl^r(\wDelta)$ (discussed in Remark~\ref{rem:HilbertPolynomial}).
\end{remark}

\begin{remark}
	Building on Remarks~\ref{rem:HilbertPolynomial} and~\ref{rem:AS8r+1}, we have observed in all the examples of Sections~\ref{exam:intro} and~\ref{sec:Examples} that $\LB(\Delta,d,r)=\dim \spl^r_d(\Delta)$ (when $\Delta$ is generic) for $d$ at least the \textit{initial degree} of $\spl^r_d(\Delta)$; that is, the bound begins to give the exact dimension of the spline space as soon as there are non-trivial splines.  To prove this one would have to know (1) that $\LB(\Delta,d,r)$ agrees with $\dim \spl^r_d(\Delta)$ for $d\gg 0$ and (2) that the regularity of $\spl^r(\wDelta)$ (see Remark~\ref{rem:HilbertPolynomial}) is very close to the initial degree of $\spl^r(\Delta)$.  We discuss (1) in Remark~\ref{rem:exact}.  We expect (2) to be quite difficult; a similar statement is not even known for generic triangulations, although we expect it to be true as we indicate in Remark~\ref{rem:exact}.
\end{remark}

\begin{remark}\label{rem:exact}
	In all of the examples in Sections~\ref{exam:intro} and~\ref{sec:Examples}, if $d\gg 0$ and $\Delta$ is generic we have $\LB(\Delta,d,r)=\dim \spl^r_d(\Delta)$; in other words $\LB(\Delta,d,r)$ is the Hilbert polynomial of $\spl^r(\wDelta)$ when $\Delta$ is generic.  This is not always the case, although it is only possible for $\LB(\Delta,d,r)$ to differ from $\dim\spl^r_d(\Delta)$ by a constant in large degree.  In fact, the only term in which we can have error is the approximation provided by Proposition~\ref{prop:AssHom} to the constant $C$ which is equal to $\dim H_2(\calR/\calJ[\wDelta])$ for $d\gg 0$.  If $\gamma$ is a boundary vertex, we see from Proposition~\ref{prop:AssHom} that its contribution to $C$ is
	\[
	\sum_{i=0}^{3r+1} \dim \hspl^r_i(\Delta_\gamma)-\LBos(\Delta,i,r)\ .
	\]
	If $\dim\hspl^r_i(\Delta_\gamma)=\max\left\{\binom{i+2}{2},\LBos(\Delta,i,r)\right\}$ for $0\le i\le 3r+1$, then this contribution coincides exactly with $\sum_{i=0}^{3r+1} \dim \left[\binom{i+2}{2}-\LBos(\Delta,i,r)\right]_+=\sum_{i=0}^{3r+1} \dim \left[-\chi(\calJ[\Delta],i)\right]_+$ and we capture the entire contribution of the boundary vertex $\gamma$ to $C$.
	
	If $\gamma$ is an interior vertex, the proof of Theorem~\ref{thm:WhitelyGenericLowDegree} in Section~\ref{sec:general_lowerbound} shows that its contributions to $C$ in degree $d\le D_\gamma$ can be accounted for; in particular the term $\dim \tJ(\gamma)_d$ for $d\le D_\gamma$ appears both in $C$ and in the Euler characteristic of $\calJ$ with opposite signs, and so it cancels.  By Propositions~\ref{prop:AssHom} and~\ref{prop:EulerCharBounds}, the contribution of $\gamma$ to $C$ in degrees $d> D_\gamma$ is
	\[
	\sum_{i=D_\gamma+1}^{3r+1} \dim \hspl^r_i(\Delta_\gamma)-\LBcs(\Delta,i,r)\ .
\]
	If $\dim\hspl^r_i(\Delta_\gamma)=\max\left\{\binom{i+2}{2},\LBcs(\Delta,i,r)\right\}$ for $i>D_\gamma$ then we again capture all of the contribution of the interior vertex $\gamma$ to $C$.
	
	This leads us to Questions~7.1 and~7.2 in~\cite{PaperA}, namely, is it typically true that
	\begin{equation}\label{eq:7.2}
	\dim \hspl^r_d(\Delta)=\max\left\{\binom{d+2}{2}, \LBos(\Delta,d,r)
	\right\}
	\end{equation}
	when $\Delta$ is a generic open vertex star, and 
	that for $d>D_\gamma$ and $\Delta$ a generic closed vertex star
	\begin{equation}\label{eq:7.1}
	\dim \hspl^r_d(\Delta)=\max\left\{\binom{d+2}{2},\LBcs(\Delta,d,r)\right\}?
	\end{equation}
	(Theorem~\ref{thm:WhitelyGenericLowDegree} shows that $\dim \hspl^r_d(\Delta)=\binom{d+2}{2}$ when $d\le D_\gamma$ and $\Delta$ is a generic closed vertex star.)
	There are configurations for open vertex stars, discussed in~\cite{PaperA}, where it is \textit{not} true that $\dim \hspl^r_d(\Delta)=\max\left\{\binom{d+2}{2}, \LBos(\Delta,d,r)\right\}$ even for generic vertex positions.  If such a configuration is present as the star of a boundary vertex inside of a larger \mesh, then our lower bound will \textit{not} give the exact dimension in large degree.  We do raise the possibility in Question~7.2 of~\cite{PaperA} that there could be finitely many \textit{sub-configurations} which serve as obstructions to the correctness of Equation~\eqref{eq:7.2} when $\Delta$ is generic.  We are not aware of any configurations where Equation~\eqref{eq:7.1} fails for generic vertex positions when $d>D_\gamma$.
\end{remark}

\begin{remark}
	The formula we give in Theorem~\ref{thm:LBGenericTet} is a lower bound for $\dim\spl^r_d(\Delta)$ for $d\gg 0$ when $\Delta$ is \textit{generic}.  That is, it only depends on purely combinatorial information of $\Delta$ such as how many triangular faces are incident upon a given edge, and not on geometric information such as whether the linear span of these triangular faces coincide.  It is well-known that such coincident linear spans cause a jump in the dimension of $\spl^r_d(\Delta)$.  As we indicate in~\cite[Example~6.2]{PaperA}, our techniques can sometimes be adjusted to improve the lower bound $\LB(\Delta,d,r)$ for these types of special positions.  We leave this as a future research direction.  Other special positions, such as the special positions of the Morgan-Scott configuration, may depend on global geometry which is invisible to our methods. 
\end{remark}

\end{document}